\documentclass[colorlinks]{article}
\pdfoutput=1
\usepackage{graphicx}

\usepackage{amsmath}
\usepackage{amsfonts}
\usepackage{amssymb}
\usepackage{amsthm}
\usepackage{mathabx}

\usepackage{stmaryrd} 
\usepackage{MnSymbol} 

\usepackage[all]{xy}
\usepackage{color}
\usepackage[protrusion=true,expansion=true]{microtype}

\usepackage[
            colorlinks=true,
            linkcolor=blue]{hyperref}

\newtheorem{thm}{Theorem}[section]
\newtheorem{cor}[thm]{Corollary}
\newtheorem{lem}[thm]{Lemma}
\newtheorem{prop}[thm]{Proposition}
\newtheorem{obs}[thm]{Observation}
\newtheorem{problem}[thm]{Problem}
\theoremstyle{definition}
\newtheorem{defn}[thm]{Definition}
\theoremstyle{remark}

\newtheorem{exm}[thm]{Example}

\numberwithin{equation}{section}
\newcommand\ZZ{\mathbb{Z}}

\newcommand\QQ{\mathbb{Q}}

\newcommand{\im}{\operatorname{Im}}

\newcommand{\ab}{\operatorname{^{ab}}}
\newcommand{\fin}{\operatorname{\!_{f.i.}}}
\newcommand{\fg}{\operatorname{\!_{f.g.}}}
\newcommand{\rank}{\operatorname{rk}}
\newcommand{\End}{\operatorname{End}}

\newcommand{\Aut}{\operatorname{Aut}}

\newcommand{\GL}{\operatorname{GL}}
\newcommand{\Fix}{\operatorname{Fix}}

\newcommand{\mcd}{\operatorname{gcd}}

\newcommand{\MP}{\operatorname{MP}}      
\newcommand{\SIP}{\operatorname{SIP}}    
\newcommand{\CIP}{\operatorname{CIP}}    
\newcommand{\FIP}{\operatorname{FIP}}    
\newcommand{\FPP}{\operatorname{FPP}}    
\newcommand{\WhP}{\operatorname{WhP}}    

\newcommand{\eq}{\ \Leftrightarrow \ }           
\newcommand{\normaleq}{ \trianglelefteqslant } 
\newcommand{\isom}{\simeq}     
\newcommand{\nisom}{\nsimeq}   
\newcommand{\tr}{^{\!\top}}    
\newcommand{\cab}{\mathcal{C}} 

\newcommand{\ti}{\mbox{type\,(I)}}
\newcommand{\tii}{\mbox{type\,(II)}}

\title{Algorithmic problems for\\free-abelian times free groups}%

\author{
\textbf{Jordi Delgado}\hyperref[authors]{$^{*}$}
 \&
\textbf{Enric Ventura}\hyperref[authors]{$^{\dag}$}
}

\date{\today}

\begin{document}

\label{top}
\maketitle


\noindent \textbf{MSC}: \texttt{20E05}, \texttt{20K01}.

\begin{abstract}
We study direct products of free-abelian and free groups with special emphasis on
algorithmic problems. After giving natural extensions of standard notions into that
family, we find an explicit expression for an arbitrary endomorphism of $\ZZ^m \times
F_n$. These tools are used to solve several algorithmic and decision problems for $\ZZ^m
\times F_n $: the membership problem, the isomorphism problem, the finite index problem,
the subgroup and coset intersection problems, the fixed point problem, and the Whitehead
problem.
\end{abstract}

\goodbreak

\tableofcontents

\goodbreak

\section*{Introduction}
\addcontentsline{toc}{section}{Introduction}
Free-abelian groups, namely $\ZZ^m$, are classical and very well known. Free groups, namely
$F_n$, are much wilder and have a much more complicated structure, but they have also been
extensively studied in the literature since more than a hundred years ago. The goal of this
paper is to investigate direct products of the form $\ZZ^m \times F_n$, namely free-abelian
times free groups. At a first look, it may seem that many questions and problems concerning
$\ZZ^m \times F_n$ will easily reduce to the corresponding questions or problems for
$\ZZ^m$ and $F_n$; and, in fact, this is the case when the problem considered is easy or
rigid enough. However, some other naive looking questions have a considerably more
elaborated answer in $\ZZ^m \times F_n$ rather than in $\ZZ^m$ or $F_n$. This is the case,
for example, when one considers automorphisms: $\Aut (\ZZ^m \times F_n )$ naturally
contains $GL_m(\ZZ) \times \Aut(F_n)$, but there are many more automorphisms other than
those preserving the factors $\ZZ^m$ and $F_n$. This fact causes potential complications
when studying problems involving automorphisms: apart from understanding the problem in
both the free-abelian and the free parts, one has to be able to control how is it affected
by the interaction between the two parts.

Another example of this phenomena is the study of intersections of subgroups. It is well
known that every subgroup of $\ZZ^m$ is finitely generated. This is not true for free
groups $F_n$ with $n\geqslant 2$, but it is also a classical result that all these groups
satisfy the Howson property: the intersection of two finitely generated subgroups is again
finitely generated. This elementary property fails dramatically in $\ZZ^m \times F_n$, when
$m\geqslant 1$ and $n\geqslant 2$ (a very easy example reproduced below, already appears
in~\cite{burns_intersection_1998} attributed to Moldavanski). Consequently, the algorithmic
problem of computing intersections of finitely generated subgroups of $\ZZ^m \times F_n$
(including the preliminary decision problem on whether such intersection is finitely
generated or not) becomes considerably more involved than the corresponding problems in
$\ZZ^m$ (just consisting on a system of linear equations over the integers) or in $F_n$
(solved by using the pull-back technique for graphs). This is one of the algorithmic
problems addressed below (see Section \ref{sec:CIP}).

Along all the paper we shall use the following notation and conventions. For $n\geqslant
1$, $[n]$ denotes the set integers $\{ 1,\ldots ,n\}$. Vectors from $\ZZ^m$ will always be
understood as row vectors, and matrices $\textbf{M}$ will always be though as linear maps
acting on the right, $\textbf{v} \mapsto \textbf{vM}$; accordingly, morphisms will always
act on the right of the arguments, $x\mapsto x\alpha$. For notational coherence, we shall
use uppercase boldface letters for matrices, and lowercase boldface letters for vectors
(moreover, if $w\in F_n$ then $\textbf{w}\in \mathbb{Z}^n$ will typically denote its
abelianization). We shall use lowercase Greek letters for endomorphisms of free groups,
$\phi \colon F_n \to F_n$, and uppercase Greek letters for endomorphisms of free-abelian
times free groups, $\Phi \colon \mathbb{Z}^m \times F_n \to \mathbb{Z}^m \times F_n$.

The paper is organized as follows. In Section~\ref{sec:ffab}, we introduce the family of
groups we are interested in, and we import there several basic notions and properties
shared by both families of free-abelian, and free groups, such as the concepts of rank and
basis, as well as the closeness property by taking subgroups. In Section~\ref{Dehn} we
remind the folklore solution to the three classical Dehn problems within our family of
groups. In the next two sections we study some other more interesting algorithmic problems:
the finite index subgroup problem in Section~\ref{fi}, and the subgroup and the coset
intersection problems in Section~\ref{sec:CIP}. In Section~\ref{sec:morphisms} we give an
explicit description of all automorphisms, monomorphisms and endomorphisms of free-abelian
times free groups which we then use in Section~\ref{sec:fix} to study the fixed subgroup of
an endomorphism, and in Section~\ref{sec:Wh} to solve the Whitehead problem within our
family of groups.

\goodbreak

\section{Free-abelian times free groups} \label{sec:ffab}

Let $T = \{ \, t_i \mid i \in I \,\}$ and $X=\{ \, x_j \mid j\in J \, \}$ be disjoint
(possibly empty) sets of symbols, and consider the group $G$ given by the presentation
 $$
G=\left \langle \, T,X \mid [T,\, T\sqcup X] \,\right \rangle,
 $$
where $[A,B]$ denotes the set of commutators of all elements from $A$ with all elements
from $B$. Calling $Z$ and $F$ the subgroups of $G$ generated, respectively, by $T$ and $X$,
it is easy to see that $Z$ is a free-abelian group with basis $T$, and $F$ is a free group
with basis $X$. We shall refer to the subgroups $Z=\langle T\rangle$ and $F=\langle
X\rangle$ as the \emph{free-abelian} and \emph{free parts} of $G$, respectively. Now, it is
straightforward to see that $G$ is the direct product of its free-abelian and free parts,
namely
\begin{equation} \label{eq:pres F x Z abreujada}
G=\left \langle \, T,X \mid [T,\, T\sqcup X] \,\right \rangle \isom Z\times F.
\end{equation}
We say that a group is \emph{free-abelian times free} if it is isomorphic to one of the
form~\eqref{eq:pres F x Z abreujada}.

It is clear that in every word on the generators $T\sqcup X$, the letters from $T$ can
freely move, say to the left, and so every element from $G$ decomposes as a product of an
element from $Z$ and an element from $F$, in a unique way. After choosing a well ordering
of the set $T$ (whose existence is equivalent to the axiom of choice), we have a natural
\emph{normal form} for the elements in $G$, which we shall write as $\mathbf{t^a} \, w$,
where $\mathbf{a}=(a_i)_i \in \bigoplus_{i\in I}\ZZ$, $ \mathbf{t^{a}}$ stands for the
(finite) product $\Pi_{i\in I} t_i^{a_i}$ (in the given order for $T$), and $w$ is a
reduced free word on $X$.

Observe that the center of the group $G$ is $Z$ unless $F$ is infinite cyclic, in which
case $G$ is abelian and so its center is the whole $G$. This exception will create some
technical problems later on.

We shall mostly be interested in the finitely generated case, i.e.\ when $T$ and $X$ are
both finite, say $I=[m]$ and $J=[n]$ respectively, with $m,n\geqslant 0$. In this case, $Z$
is the free-abelian group of rank $m$, $Z=\ZZ^m$, $F$ is the free group of rank $n$,
$F=F_n$, and our group $G$ becomes
\begin{equation} \label{eq:pres F_n x Z^m}
G=\ZZ^m \times F_n =\langle \, t_1,\ldots,t_m,\, x_1,\ldots,x_n \mid t_it_j=t_jt_i,\, t_ix_k=x_kt_i \, \rangle,
\end{equation}
where $i,j\in [m]$ and $k\in [n]$. The normal form for an element $g\in G$ is now
$$
g=\mathbf{t^a} \, w = t_1^{a_1} \cdots \,t_m^{a_m} \, w(x_1,\ldots ,x_n) ,
$$
where $\mathbf{a}=(a_1,\ldots ,a_m)\in \ZZ^m$ is a row integral vector, and $w=w(x_1,\ldots
,x_n)$ is a reduced free word on the alphabet $X$. Note that the symbol $\mathbf{t}$ by
itself has no real meaning, but it allows us to convert the notation for the abelian group
$\ZZ^m$ from additive into multiplicative, by moving up the vectors (i.e.\ the entries of
the vectors) to the level of exponents; this will be especially convenient when working in
$G$, a noncommutative group in general.

Observe that the ranks of the free-abelian and free parts of $G$, namely $m$ and $n$, are
not invariants of the group $G$, since $\ZZ^m \times F_1\isom \ZZ^{m+1}\times F_0$.
However, as one may expect, this is the only possible redundancy and so, we can generalize
the concepts of rank and basis from the free-abelian and free contexts to the mixed
free-abelian times free situation.

\goodbreak

\begin{obs} \label{prop:caract Fn x Z^m}
Let $Z$ and $Z'$ be arbitrary free-abelian groups, and let $F$ and $F'$ be arbitrary free
groups. If $F$ and $F'$ are not infinite cyclic, then
 $$
Z\times F\isom Z'\times F' 
\,\, \Leftrightarrow \,\, \rank(Z)=\rank(Z') \text{ and } \rank(F)=\rank(F').
 $$
\end{obs}

\begin{proof}
It is straightforward to see that the center of $Z\times F$ is $Z$ (here is where $F\nisom
\mathbb{Z}$ is needed). On the other hand, the quotient by the center gives $(Z\times F)/Z
\isom F$. The result follows immediately.
\end{proof}

\begin{defn} \label{def:rang}
Let $G=Z\times F$ be a free-abelian times free group and assume, without loss of
generality, that $F\not\isom \mathbb{Z}$. Then, according to the previous observation, the
pair of cardinals $(\kappa,\, \varsigma)$, where $\kappa$ is the abelian rank of $Z$ and
$\varsigma$ is the rank of $F$, is an invariant of $G$, which we shall refer to as the
\emph{rank} of $G$, $\rank(G)$. (We allow this abuse of notation because the rank of $G$ in
the usual sense, namely the minimal cardinal of a set of generators, is precisely
$\kappa+\varsigma$: $G$ is, in fact, generated by a set of $\kappa+\varsigma$ elements and,
abelianizing, we get $G \ab =(Z\times F)\ab=Z\oplus F\ab$, a free-abelian group of rank
$\kappa+\varsigma$, so $G$ cannot be generated by less than $\kappa+\varsigma$ elements.)
\end{defn}

\begin{defn}\label{def:base}
Let $G=Z\times F$ be a free-abelian times free group. A pair $(A,B)$ of subsets of $G$ is
called a \emph{basis} of $G$ if the following  three conditions are satisfied:
\begin{enumerate}
\item [(i)] $A$ is an abelian basis of the center of $G$,
\item [(ii)] $B$ is empty, or a free basis of a non-abelian free subgroup of $G$ (note
    that this excludes the possibility $|B|=1$),
\item [(iii)] $\langle A \cup B\rangle= G$.
\end{enumerate}
In this case we shall also say that $A$ and $B$ are, respectively, the \emph{free-abelian}
and \emph{free} components of $(A,B)$. From (i), (ii) and (iii) it follows immediately
\begin{enumerate}
\item [(iv)] $\langle A\rangle \cap \langle B\rangle = \{ 1 \}$,
\item [(v)] $A\cap B=\emptyset$,
\end{enumerate}
since $\langle A \rangle \cap \langle B \rangle$ is contained in the center of $G$, but no
non trivial element of $\langle B\rangle$ belongs to it.

\goodbreak

Usually, we shall abuse notation and just say that $A\cup B$ \emph{is a basis} of $G$. Note
that no information is lost because we can retrieve $A$ as the elements in $A\cup B$ which
belong to the center of $G$, and $B$ as the remaining elements.
\end{defn}

Observe that, by (i), (iii) and (iv) in the previous definition, if $(A,B)$ is a basis of a
free-abelian times free group $G$, then $G=\langle A\rangle \times \langle B\rangle$; and
by (i) and (ii), $\langle A\rangle$ is a free-abelian group and $\langle B\rangle$ is a
free group not isomorphic to $\mathbb{Z}$; hence, by Observation~\ref{prop:caract Fn x
Z^m}, $\rank(G)=(|A|,\, |B|)$. In particular, this implies that $(|A|,\, |B|)$ does not
depend on the particular basis $(A,B)$ chosen.

On the other hand, the first obvious example is $T\cup X$ being a basis of the group
$G=\langle T,X\mid [T,\, T\sqcup X]\rangle$ (note that if $|X|\neq 1$ then $A=T$ and $B=X$,
but if $|X|=1$ then $A=T\cup X$ and $B=\emptyset$ due to the technical requirement in
Observation~\ref{prop:caract Fn x Z^m}). We have proved the following.

\begin{cor}
Every free-abelian times free group $G$ has bases and, every basis $(A,B)$ of $G$ satisfies
$\rank(G)=(|A|,\, |B|)$. \qed
\end{cor}

\goodbreak

Let us focus now our attention to subgroups. It is very well known that every subgroup of a
free-abelian group is free-abelian; and every subgroup of a free group is again free. These
two facts lead, with a straightforward argument, to the same property for free-abelian
times free groups (this will be crucial for the rest of the paper).

\begin{prop}\label{prop:subgs de Fn x Z^m}
The family of free-abelian times free groups is closed under taking subgroups.
\end{prop}

\begin{proof}
Let $T$ and $X$ be arbitrary disjoint sets, let $G$ be the free-abelian times free group
given by presentation~(\ref{eq:pres F x Z abreujada}), and let $H\leqslant G$.

If $|X|=0,1$ then $G$ is free-abelian, and so $H$ is again free-abelian (with rank less
than or equal to that of $G$); the result follows.

\goodbreak

Assume $|X|\geqslant 2$. Let $Z=\langle T\rangle$ and $F=\langle X\rangle$ be the
free-abelian and free parts of $G$, respectively, and let us consider the natural short
exact sequence associated to the direct product structure of $G$:
 $$
1\longrightarrow Z \overset{\iota}{\longrightarrow} Z \times F =G \overset{\pi}{\longrightarrow} F \longrightarrow 1,
 $$
where $\iota$ is the inclusion, $\pi$ is the projection $\mathbf{t^{a}}w\mapsto w$, and
therefore $\ker(\pi)=Z=\im(\iota)$. Restricting this short exact sequence to $H\leqslant
G$, we get
 $$
1\longrightarrow \ker(\pi_{\mid H}) \overset{\iota}{\longrightarrow} H\overset{\pi_{\mid H}}{\longrightarrow} H\pi
\longrightarrow 1,
 $$
where $1\leqslant \ker(\pi_{\mid H})=H\cap \ker(\pi)=H\cap Z\leqslant Z$, and $1\leqslant
H\pi \leqslant F$.
Therefore, $\ker(\pi_{\mid H})$ is a free-abelian group, and $H\pi$ is a
free group. Since $H\pi$ is free, $\pi_{\mid H}$ has a splitting
\begin{equation} \label{eq:escissio alpha}
H\overset{\alpha}{\longleftarrow} H\pi,
\end{equation}
sending back each element of a chosen free basis for $H\pi$ to an arbitrary preimage.

\goodbreak

Hence, $\alpha$ is injective, $H\pi\alpha\leqslant H$ is isomorphic to $H\pi$, and
straightforward calculations show that the following map is an isomorphism:
\begin{equation} \label{eq:iso factoritzacio subgrup}
\begin{array}{rcl}
\Theta_{\alpha} \colon H & \longrightarrow & \ker(\pi_{\mid H}) \times H\pi\alpha \\ h & \longmapsto & \bigl(h(h \pi
\alpha)^{-1},\, h\pi\alpha \bigr).
\end{array}
\end{equation}
Thus $H\isom \ker(\pi_{\mid H}) \times H\pi\alpha$ is free-abelian times free and the
result is proven.
\end{proof}

This proof shows a particular way of decomposing $H$ into a direct product of a
free-abelian subgroup and a free subgroup, which depends on the chosen splitting $\alpha$,
namely
\begin{equation}\label{eq:factoritzacio subgrup}
H=(H\cap Z)\times H\pi\alpha.
\end{equation}
We call the subgroups $H\cap Z$ and $H\pi\alpha$, respectively, the \emph{free-abelian} and
\emph{free parts of $H$, with respect to the splitting $\alpha$}. Note that the
free-abelian and free parts of the subgroup $H=G$ with respect to the natural inclusion
$G\hookleftarrow F \colon \alpha$ coincide with what we called the free-abelian and free
parts of $G$.

Furthermore, Proposition~\ref{prop:subgs de Fn x Z^m} and the
decomposition~\eqref{eq:factoritzacio subgrup} give a characterization of the bases, rank,
and all possible isomorphism classes of such an arbitrary subgroup $H$.

\begin{cor}\label{cor:caracteritzacio base combinatoria}\index{base!caracteritzacions}
With the above notation, a subset $E \subseteq H\leqslant G=Z\times F$ is a basis of $H$ if
and only if
 $$
E=E_Z \sqcup E_F,
 $$
where $E_Z$ is an abelian basis of $H \cap Z$, and $E_F$ is a free basis of $H\pi \alpha$,
for a certain splitting $\alpha$ as in~\eqref{eq:escissio alpha}.
\end{cor}

\begin{proof}
The implication to the left is straightforward, with $E=A\sqcup B$, and $(A,B)=(E_Z, E_F)$
except for the case $\rank(F)=1$, when we have $(A,B)=(E_Z\sqcup E_F, \emptyset )$.

Suppose now that $E=A\sqcup B$ is a basis of $H$ in the sense of
Definition~$\ref{def:base}$, and let us look at the decomposition~\eqref{eq:factoritzacio
subgrup}, for suitable $\alpha$. If $\rank(H\pi)=1$, then $H$ is abelian, $A$ is an abelian
basis for $H$, $B=\emptyset$ and all but exactly one of the elements in $A$ belong to
$H\cap Z$ (i.e.\ have normal forms using only letters from $T$); in this case the result is
clear, taking $E_F$ to be just that special element. Otherwise, $Z(H)=H\cap Z$ having $A$
as an abelian basis; take $E_Z =A$ and $E_F=B$. It is clear that the projection $\pi\colon
H \twoheadrightarrow H\pi,\ \mathbf{t^{a}} u \mapsto u$, restricts to an isomorphism
$\pi|_{\langle B\rangle}\colon \langle B\rangle \to H\pi$ since no nontrivial element in
$\langle B\rangle$ belong to $\ker \pi =H\cap Z$. Therefore, taking $\alpha =\pi|_{\langle
B\rangle}^{-1}$, $E_F$ is a free basis of $H\pi \alpha$.
\end{proof}

\begin{cor} \label{cor:classes isomorfia subgrups}
Let $G$ be the free-abelian times free group given by presentation~\eqref{eq:pres F x Z
abreujada}, and let $\rank(G)=(\kappa,\, \varsigma )$. Every subgroup $H\leqslant G$ is
again free-abelian times free with $\rank(H)=(\kappa',\, \varsigma' )$ where,
\begin{itemize}
\item[\emph{(i)}] in case of $\varsigma=0$: $0\leqslant \kappa'\leqslant \kappa$ and
    $\varsigma' =0$;
\item[\emph{(ii)}] in case of $\varsigma \geqslant 2$: either $0\leqslant
    \kappa'\leqslant \kappa +1$ and $\varsigma' =0$, or $0\leqslant \kappa'\leqslant
    \kappa$ and $0\leqslant \varsigma' \leqslant \max \{ \varsigma,\, \aleph_0 \}$ and
    $\varsigma'\neq 1$.
\end{itemize}
Furthermore, for every such $(\kappa',\, \varsigma' )$, there is a subgroup $H\leqslant G$
such that $\rank(H)=(\kappa',\, \varsigma' )$. \qed
\end{cor}

\goodbreak

Along the rest of the paper, we shall concentrate on the finitely generated case. From
Proposition~\ref{prop:subgs de Fn x Z^m} we can easily deduce the following corollary,
which will be useful later.

\begin{cor}\label{cor:H fg sii Hpi fg}
A subgroup $H$ of $\ZZ^m \times F_n$ is finitely generated if and only if its projection to
the free part $H\pi$ is finitely generated. \qed
\end{cor}

The proof of Proposition~\ref{prop:subgs de Fn x Z^m}, at least in the finitely generated
case, is completely algorithmic; i.e. if $H$ is given by a finite set of generators, one
can effectively choose a splitting $\alpha$, and compute a basis of the free-abelian and
free parts of $H$ (w.r.t. $\alpha$). This will be crucial for the rest of the paper, and we
make it more precise in the following proposition.

\begin{prop} \label{prop:bases algorismiques}
Let $G=\ZZ^m \times F_n$ be a finitely generated free-abelian times free group. There is an
algorithm which, given a subgroup $H\leqslant G$ by a finite family of generators, it
computes a basis for $H$ and writes both, the new elements in terms of the old generators,
and the old generators in terms of the new basis.
\end{prop}

\begin{proof}
If $n=|X|=0,1$ then $G$ is free-abelian and the problem is a straightforward exercise in
linear algebra. So, let us assume $n\geqslant 2$.

We are given a finite set of generators for $H$, say $\mathbf{t}^\mathbf{c_1} w_1, \ldots,
\mathbf{t}^\mathbf{c_{p}} w_{p}$, where $p\geqslant 1$, $\mathbf{c_1}, \ldots,
\mathbf{c_{p}} \in \ZZ^m$ are row vectors, and $w_1, \ldots, w_{p}\in F_n$ are reduced
words on $X=\{ x_1, \ldots ,x_n \}$. Applying suitable Nielsen transformations,
see~\cite{lyndon_combinatorial_2001}, we can algorithmically transform the $p$-tuple
$(w_1,\ldots ,w_p)$ of elements from $F_n$, into another of the form $(u_1,\ldots
,u_{n'},1,\ldots ,1)$, where $\{ u_1,\ldots ,u_{n'}\}$ is a free basis of $\langle
w_1,\ldots ,w_p\rangle =H\pi$, and $0\leqslant n'\leqslant p$. Furthermore, reading along
the Nielsen process performed, we can effectively compute expressions of the new elements
as words on the old generators, say $u_j=\eta_j (w_1,\ldots ,w_p)$, $j\in [n']$, as well as expressions of the old
generators in terms of the new free basis, say $w_i =\nu_i(u_1,\ldots ,u_{n'})$, for $i\in
[p]$.

Now, the map $\alpha \colon H\pi \to H$, $u_j \mapsto \eta_j(\mathbf{t}^\mathbf{c_1} w_1,
\ldots , \mathbf{t}^\mathbf{c_{p}} w_{p})$ can serve as a splitting in the proof of
Proposition~\ref{prop:subgs de Fn x Z^m}, since $\eta_j(\mathbf{t}^\mathbf{c_1} w_1, \ldots
, \mathbf{t}^\mathbf{c_{p}} w_{p})=\mathbf{t}^{\mathbf{a_j}}\eta_j (w_1,\ldots
,w_p)=\mathbf{t}^{\mathbf{a_j}}u_j \in H$, where $\mathbf{a_j}$, $j\in [n']$, are integral
linear combinations of $\mathbf{c_1},\ldots ,\mathbf{c_p}$.

It only remains to compute an abelian basis for $\ker(\pi_{\mid H})=H\cap \ZZ^m$. For each
one of the given generators $h=\mathbf{t}^\mathbf{c_i} w_i$, compute
$h(h\pi\alpha)^{-1}=\mathbf{t}^{\mathbf{d_i}}$ (here, we shall need the words $\nu_i$
computed before). Using the isomorphism $\Theta_{\alpha}$ from the proof of
Proposition~\ref{prop:subgs de Fn x Z^m}, we deduce that $\{
\mathbf{t}^{\mathbf{d_1}},\ldots ,\mathbf{t}^{\mathbf{d_p}}\}$ generate $H\cap \ZZ^m$; it
only remains to use a standard linear algebra procedure, to extract from here an abelian
basis $\{\mathbf{t}^{\mathbf{b_1}},\ldots ,\mathbf{t}^{\mathbf{b_{m'}}} \}$ for $H\cap
\ZZ^m$. Clearly, $0\leqslant m'\leqslant m$.

We immediately get a basis $(A,B)$ for $H$ (with just a small technical caution): if
$n'\neq 1$, take $A=\{ \mathbf{t}^\mathbf{b_1}, \ldots , \mathbf{t}^\mathbf{b_{m'}} \}$ and
$B=\{ \mathbf{t}^\mathbf{a_1} u_1, \ldots , \mathbf{t}^\mathbf{a_{n'}} u_{n'} \}$; and if
$n'=1$ take $A=\{ \mathbf{t}^\mathbf{b_1}, \ldots , \mathbf{t}^\mathbf{b_{m'}},\,
\mathbf{t}^\mathbf{a_1} u_1\}$ and $B=\emptyset$.

On the other hand, as a side product of the computations done, we have the expressions
$\mathbf{t}^{\mathbf{a_j}}u_j= \eta_j(\mathbf{t}^\mathbf{c_1} w_1, \ldots ,
\mathbf{t}^\mathbf{c_{p}} w_{p})$, $j\in [n']$. And we can also compute expressions of the
$\mathbf{t}^{\mathbf{b_i}}$'s in terms of the $\mathbf{t}^{\mathbf{d_i}}$'s, and of the
$\mathbf{t}^{\mathbf{d_i}}$'s in terms of the $\mathbf{t}^\mathbf{c_i} w_i$'s. Hence we can
compute expressions for each one of the new elements in terms of the old generators.

For the other direction, we also have the expressions $w_i =\nu_i(u_1,\ldots ,u_{n'})$, for
$i\in [p]$. Hence, $\nu_i(\mathbf{t}^\mathbf{a_1} u_1, \ldots , \mathbf{t}^\mathbf{a_{n'}}
u_{n'})=\mathbf{t}^{\mathbf{e_i}}w_i$ for some $\mathbf{e_i}\in \ZZ^m$. But $H\ni
(\mathbf{t}^{\mathbf{c_i}}w_i)
(\mathbf{t}^{\mathbf{e_i}}w_i)^{-1}=\mathbf{t}^{\mathbf{c_i-e_i}}\in \ZZ^m$,  so we can
compute integers $\lambda_1,\ldots ,\lambda_{m'}$ such that $\mathbf{c_i}-\mathbf{e_i}
=\lambda_1 \mathbf{b_1}+\cdots +\lambda_{m'}\mathbf{b_{m'}}$. Thus,
$\mathbf{t}^{\mathbf{c_i}}w_i =\mathbf{t}^{\mathbf{c_i}-\mathbf{e_i}}
\mathbf{t}^{\mathbf{e_i}}w_i = \mathbf{t}^{\lambda_1 \mathbf{b_1}+\cdots
+\lambda_{m'}\mathbf{b_{m'}}} \mathbf{t}^{\mathbf{e_i}} w_i =
(\mathbf{t}^{\mathbf{b_1}})^{\lambda_1}\cdots (\mathbf{t}^{\mathbf{b_{m'}}})^{\lambda_{m'}}
\nu_i(\mathbf{t}^\mathbf{a_1} u_1, \ldots , \mathbf{t}^\mathbf{a_{n'}} u_{n'})$, for $i\in
[p]$.
\end{proof}

As a first application of Proposition~\ref{prop:bases algorismiques}, free-abelian times
free groups have solvable \emph{membership problem}. Let us state it for an arbitrary group
$G$.

\begin{problem}[\textbf{Membership Problem, $\MP(G)$}]
Given elements $g,\, h_1,\ldots ,h_p \in G$, decide whether $g\in H=\langle h_1,\ldots
,h_p\rangle$ and, in this case, computes an expression of $g$ as a word on the~$h_i$'s.
\end{problem}

\goodbreak

\begin{prop} \label{lem:Membership problem per Fn x Z^m} \index{membership problem!decidibilitat i c\`{a}lcul}
The Membership Problem for $G=\ZZ^m \times F_n$ is solvable.
\end{prop}

\begin{proof}
Write $g=\mathbf{t}^{\mathbf{a}}w$. We start by computing a basis for $H$ following
Proposition~\ref{prop:bases algorismiques}, say $\{ \mathbf{t}^\mathbf{b_1}, \ldots
,\mathbf{t}^\mathbf{b_{m'}},\, \mathbf{t}^\mathbf{a_1} u_1, \ldots
,\mathbf{t}^\mathbf{a_{n'}} u_{n'}\}$. Now, check whether $g\pi =w\in H\pi=\langle
u_1,\ldots ,u_{n'}\rangle$ (membership is well known for free groups). If the answer is
negative then $g\not\in H$ and we are done. Otherwise, a standard algorithm for membership
in free groups gives us the (unique) expression of $w$ as a word on the $u_j$'s, say
$w=\omega (u_1,\ldots ,u_{n'})$. Finally, compute $\omega (\mathbf{t}^\mathbf{a_1} u_1,
\ldots ,\mathbf{t}^\mathbf{a_{n'}} u_{n'})=\mathbf{t}^{\mathbf{c}}w \in H$. It is clear
that $\mathbf{t}^{\mathbf{a}}w\in H$ if and only if $\mathbf{t}^{\mathbf{a}-\mathbf{c}}
=(\mathbf{t}^{\mathbf{a}}w)(\mathbf{t}^{\mathbf{c}}w)^{-1}\in H$ that is, if and only if
$\mathbf{a}-\mathbf{c} \in \langle \mathbf{b_1}, \ldots ,\mathbf{b_{m'}}\rangle \leqslant
\ZZ^m$. This can be checked by just solving a system of linear equations; and, in the
affirmative case, we can easily find an expression for $g$ in terms of $\{
\mathbf{t}^\mathbf{b_1}, \ldots ,\mathbf{t}^\mathbf{b_{m'}},\, \mathbf{t}^\mathbf{a_1} u_1,
\ldots ,\mathbf{t}^\mathbf{a_{n'}} u_{n'}\}$, like at the end of the previous proof.
Finally, it only remains to convert this into an expression of $g$ in terms of $\{
h_1,\ldots ,h_p \}$ using the expressions we already have for the basis elements in terms
of the $h_i$'s.
\end{proof}

\begin{cor}\label{cor:membership}
The membership problem for arbitrary free-abelian times free groups is solvable.
\end{cor}

\begin{proof}
We have $G=Z\times F$, where $Z=\langle T\rangle$ is an arbitrary free-abelian group and
$F=\langle X\rangle$ is an arbitrary free group. Given elements $g,\, h_1,\ldots ,h_p \in
G$, let $\{ t_1,\ldots ,t_m\}$ (resp. $\{ x_1,\ldots ,x_n \}$) be the finite set of letters
in $T$ (resp. in $X$) used by them. Obviously all these elements, as well as the subgroup
$H=\langle h_1,\ldots ,h_p\rangle$, live inside $\langle t_1,\ldots ,t_m\rangle \times
\langle x_1,\ldots ,x_n\rangle \isom \ZZ^m \times F_n$ and we can restrict our attention to
this finitely generated environment. Proposition~\ref{lem:Membership problem per Fn x Z^m}
completes the proof.
\end{proof}

\goodbreak

To conclude this section, let us introduce some notation that will be useful later. Let $H$
be a finitely generated subgroup of $G=\ZZ^m \times F_n$, and consider a basis for $H$,
\begin{equation}
\{\mathbf{t}^\mathbf{b_1}, \ldots ,\mathbf{t}^\mathbf{b_{m'}} ,\, \mathbf{t}^\mathbf{a_1} u_1, \ldots ,
\mathbf{t}^\mathbf{a_{n'}} u_{n'} \},
\end{equation}
where $0\leqslant m'\leqslant m$, $\{ \mathbf{b_1}, \ldots ,\mathbf{b_{m'}}\}$ is an
abelian basis of $H\cap \ZZ^m \leqslant \ZZ^m$, ${0\leqslant n'}$, $\mathbf{a_1}, \ldots,
\mathbf{a_{n'}} \in \ZZ^m$, and $\{ u_1, \ldots, u_{n'} \}$ is a free basis of $H\pi
\leqslant F_n$. Let $L=\langle \mathbf{b_1}, \ldots, \mathbf{b_{m'}}\rangle \leqslant
\ZZ^m$ (with additive notation, i.e.\ these are true vectors with $m$ integral coordinates
each), and let us denote by $\mathbf{A}$ the $n'\times m$ integral matrix whose rows are
the $\mathbf{a_i}$'s,
 $$
\mathbf{A}=\left( \begin{array}{c} \mathbf{a_{1}} \\ \vdots \\ \mathbf{a_{n'}} \end{array}
\right ) \in \mathcal{M}_{n' \times m}(\ZZ).
 $$
If $\omega$ is a word on $n'$ letters (i.e.\ an element of the abstract free group
$F_{n'}$), we will denote by $\omega(u_1,\ldots,u_{n'})$ the element of $H\pi$ obtained by
replacing the $i$-th letter in $\omega$ by $u_i$, $i\in [n']$. And we shall use boldface,
$\boldsymbol\omega$, to denote the abstract abelianization of $\omega$, which is an
integral vector with $n'$ coordinates, $\boldsymbol\omega \in \ZZ^{n'}$ (not to be confused
with the image of $\omega(u_1,\ldots ,u_{n'})\in F_n$ under the abelianization map $F_n
\twoheadrightarrow \ZZ^{n}$). Straightforward calculations provide the following result.

\begin{lem} \label{lem:descripcio d'un subgrup fg en termes d'una base}
With the previous notations, we have
\begin{equation*}
H=\{  \mathbf{t}^{\mathbf{a}} \, \omega (u_1, \ldots, u_{n'}) \mid \omega \in F_{n'} ,\mathbf{a} \in \boldsymbol\omega \mathbf{A} + L  \},
\end{equation*}
a convenient description of $H$. \qed
\end{lem}

\begin{defn}\label{def:complecio abeliana}
Given a subgroup $H\leqslant \ZZ^m \times F_n$, and an element $w\in F_n$, we define the
\emph{abelian completion of $w$ in $H$} as
 $$
\cab_{w,H} = \{ \mathbf{a}\in \ZZ^m \mid t^{\mathbf{a}} w \in H\} \subseteq \ZZ^m.
 $$
\end{defn}

\begin{cor}~\label{cor:propietats complecio abeliana}
With the above notation, for every $w\in F_n$ we have
\begin{enumerate}
\item [\emph{(i)}] if $w \not\in H \pi$, then $\cab_{w,H} = \emptyset$,
\item [\emph{(ii)}] if $w \in H \pi$, then $\cab_{w,H} = \boldsymbol\omega \mathbf{A} +
    L$, where $\boldsymbol\omega$ is the abelianization of the word $\omega$ which
    expresses $w\in F_n$ in terms of the free basis $\{ u_1,\ldots,u_{n'} \}$ (i.e.\
    $w=\omega(u_1,\ldots,u_{n'})$; note the difference between $w$ and $\omega$).
\end{enumerate}
Hence, $\cab_{w,H} \subseteq \ZZ^m$ is either empty or an affine variety with direction $L$
(i.e.\ a coset of $L$). \qed
\end{cor}

\goodbreak

\section{The three Dehn problems}\label{Dehn}

We shall dedicate the following sections to solve several algorithmic problems in $G=\ZZ^m
\times F_n$. The general scheme will be reducing the problem to the analogous problem on
each part, $\ZZ^m$ and $F_n$, and then apply the vast existing literature for free-abelian
and free groups. In some cases, the solutions for the free-abelian and free parts will
naturally build up a solution for $G$, while in some others the interaction between both
will be more intricate and sophisticated; everything depends on how complicated becomes the
relation between the free-abelian and free parts, with respect to the problem.

From the algorithmic point of view, the statement ``let $G$ be a group'' is not
sufficiently precise. The algorithmic behavior of $G$ may depend on how it is given to us.
For free-abelian times free groups, we will always assume that they are finitely generated
and given to us with the standard presentation~(\ref{eq:pres F_n x Z^m}). We will also
assume that the elements, subgroups, homomorphisms and other objects associated with the
group are given to us in terms of this presentation.

As a first application of the existence and computability of bases for finitely generated
subgroups of $G$, we already solved the membership problem (see
Corollary~\ref{cor:membership}), which includes the word problem. This last one, together
with the conjugacy problem, are quite elementary because of the existence of
algorithmically computable normal forms for the elements in $G$. The third of Dehn's
problems is also easy within our family of groups.

\begin{prop}
Let $G=\ZZ^m \times F_n$. Then
\begin{itemize}
\item[\emph{(i)}] the word problem for $G$ is solvable,
\item[\emph{(ii)}] the conjugacy problem for $G$ is solvable,
\item[\emph{(iii)}] the isomorphism problem is solvable within the family of finitely
    generated free-abelian times free groups.
\end{itemize}
\end{prop}

\begin{proof}
As seen above, every element from $G$ has a normal form, easily computable from an
arbitrary expression in terms of the generators. Once in normal form,
$\mathbf{t}^{\mathbf{a}}u$ equals 1 if and only if $\mathbf{a}=\mathbf{0}$ and $u$ is the
empty word. And $\mathbf{t}^{\mathbf{a}}u$ is conjugate to $\mathbf{t}^{\mathbf{b}}v$ if
and only if $\mathbf{a}=\mathbf{b}$ and $u$ and $v$ are conjugate in $F_n$. This solves the
word and conjugacy problems in $G$.

For the isomorphism problem, let $\langle X \mid R \, \rangle$ and $\langle Y \mid S \,
\rangle$ be two arbitrary finite presentations of free-abelian times free groups $G$ and
$G'$ (i.e.\ we are given two arbitrary finite presentations plus the information that both
groups are free-abelian times free). So, both $G$ and $G'$ admit presentations of the
form~\eqref{eq:pres F_n x Z^m}, say $\mathcal{P}_{n,m}$ and $\mathcal{P}_{n',m'}$, for some
integers $m,n,m',n'\geqslant 0$, $n,n'\neq 1$ (unknown at the beginning). It is well known
that two finite presentations present the same group if and only if they are connected by a
finite sequence of Tietze transformations (see~\cite{lyndon_combinatorial_2001}); so, there
exist finite sequences of Tietze transformations, one from $\langle X \mid R \,\rangle$ to
$\mathcal{P}_{n,m}$, and another from $\langle Y \mid S \,\rangle$ to $\mathcal{P}_{n',m'}$
(again, unknown at the beginning). Let us start two diagonal procedures exploring,
respectively, the tree of all possible Tietze transformations successively aplicable to
$\langle X \mid R \, \rangle$ and $\langle Y \mid S \, \rangle$. Because of what was just
said above, both procedures will necessarily reach presentations of the desired form in
finite time. When knowing the parameters $m,n,m',n'$ we apply Observation~\ref{prop:caract
Fn x Z^m} and conclude that $\langle X \mid R \,\rangle$ and $\langle Y \mid S \,\rangle$
are isomorphic if and only if $n=n'$ and $m=m'$. (This is a brute force algorithm, very far
from being efficient from a computational point of view.)
\end{proof}

\goodbreak

\section{Finite index subgroups}\label{fi}

In this section, the goal is to find an algorithm solving the Finite Index Problem in a
free-abelian times free group $G$:

\begin{problem}[\textbf{Finite Index Problem, $\FIP(G)$}]
Given a finite list $w_1,\ldots,w_s$ of elements in $G$, decide whether the subgroup
$H=\langle w_1,\ldots,w_s \rangle$ is of finite index in $G$ and, if so, compute the index an
a system of right (or left) coset representatives for $H$.
\end{problem}

To start, we remind that this same algorithmic problem is well known to be solvable both
for free-abelian and for free groups. Given several vectors $\mathbf{w_1},\ldots
,\mathbf{w_s}\in \ZZ^m$, the subgroup $H=\langle \mathbf{w_1},\ldots ,\mathbf{w_s}\rangle$
is of finite index in $\ZZ^m$ if and only if it has rank $m$. And here is an algorithm to
make such a decision, and (in the affirmative case) to compute the index $[\ZZ^m : H]$ and a set of coset
representatives for $H$: consider the $s\times m$ integral matrix
$\mathbf{W}$ whose rows are the $\mathbf{w_i}$'s, and compute its Smith normal form, i.e.\
$\mathbf{PW}=\operatorname{\mathbf{diag}}(d_1, d_2, \ldots ,d_r,0,\ldots ,0)\mathbf{Q}$,
where $\mathbf{P}\in \GL_s(\ZZ)$, $\mathbf{Q}\in \GL_m(\ZZ)$, $d_1,\ldots ,d_r$ are
non-zero integers each dividing the following one, $d_1 \divides  d_2 \divides \cdots
\divides d_r \neq 0$, the diagonal matrix has size $s\times m$, and
$r=\rank(\mathbf{W})\leqslant \min \{s,m\}$ (fast algorithms are well known to compute all
these from $\mathbf{W}$, see~\cite{artin_algebra_2010} for details). Now, if $r<m$ then
$[\ZZ^m : H]=\infty$ and we are done. Otherwise, $H$ is the subgroup generated by the rows
of ($\textbf{W}$ and so those of) $\mathbf{PW}$, i.e.\ the image under the automorphism
$Q\colon \ZZ^m \to \ZZ^m$, $\mathbf{v}\mapsto \mathbf{vQ}$ of the subgroup $H'$ generated
by the simple vectors $(d_1,0,\ldots ,0), \ldots ,(0,\ldots ,0,d_m)$. It is clear that
$[\ZZ^m : H']=d_1d_2\cdots d_m$, with $\{ (r_1, \ldots , r_m) \mid r_i \in [d_i]\}$ being a
set of coset representatives for $H'$. Hence, $[\ZZ^m : H]=d_1d_2\cdots d_m$ as well, with
$\{ (r_1, \ldots , r_m)\mathbf{Q} \mid r_i \in [d_i]\}$ being a set of coset
representatives for $H$.

On the other hand, the subgroup $H=\langle w_1,\ldots,w_s \rangle\leqslant F_n$ has finite
index if and only if every vertex of the core of the Schreier graph for $H$, denoted
$\mathcal{S}(H)$, are complete (i.e.\ have degree $2n$); this is algorithmically checkable
by means of fast algorithms. And, in this case, the labels of paths in a chosen maximal
tree $T$ from the basepoint to each vertex (resp. from each vertex to the basepoint) give a
set of left (resp. right) coset representatives for $H$, whose index in $F_n$ is then the
number of vertices of $\mathcal{S}(H)$. For details, see~\cite{stallings_topology_1983} for
the classical reference or \cite{kapovich_stallings_2002} for a more modern and
combinatorial approach.

Hence, $\FIP(\mathbb{Z}^m)$ and $\FIP(F_n)$ are solvable. In order to build an algorithm to
solve the same problem in $\ZZ^m\times F_n$, we shall need some well known basic facts
about indices of subgroups that we state in the following two lemmas. For a subgroup
$H\leqslant G$ of an arbitrary group $G$, we will write~${H\leqslant\fin G}$ to denote
$[G:H]<\infty$.

\begin{lem} \label{lem:index de subgrups a traves de epimorfismes}
Let $G$ and $G'$ be arbitrary groups, $\rho \colon G\twoheadrightarrow G'$ an epimorphism
between them, and let $H\leqslant G$ and $H'\leqslant G'$ be arbitrary subgroups. Then,
\begin{itemize}
\item[\emph{(i)}] $[G':H\rho]\leqslant [G:H]$; in particular, if $H\leqslant\fin G$ then
    ${H\rho \leqslant\fin G'}$.
\item [\emph{(ii)}] $[G':H']=[G:H'\rho^{-1}]$; in particular, $H'\leqslant\fin G'$ if and
    only if ${H'\rho^{-1} \leqslant\fin G}$. \qed
\end{itemize}
\end{lem}

\begin{lem} \label{lem:index producte directe}
Let $Z$ and $F$ be arbitrary groups, and let $H\leqslant Z\times F$ be a subgroup of their
direct product. Then
 $$
[Z\times F:H]\leqslant [Z:H\cap Z]\cdot [F:H\cap F] ,
 $$
and
 $$
H\leqslant\fin Z\times F\eq H\cap Z\leqslant\fin Z \text{ and } H\cap F\leqslant\fin F.
 $$
\end{lem}

\begin{proof}
It is straightforward to check that the map
\begin{equation} \label{eq:index producte directe}
\begin{array}{rcl}
Z/(H\cap Z)\ \times \ F/(H\cap F) & \rightarrow & (Z\times F)/H \\[3pt]
\bigl( z\cdot (H\cap Z)\, ,\, f\cdot (H\cap F) \bigr) & \mapsto & zf\cdot H
\end{array}
\end{equation}
is well defined and onto; the inequality and one implication follow immediately. The other
implication is a well know fact.
\end{proof}

Let $G=\ZZ^m \times F_n$, and let $H$ be a subgroup of $G$. If $H\leqslant\fin G$ then,
applying Lemma~\ref{lem:index de subgrups a traves de epimorfismes}~(i) to the canonical
projections $\tau \colon G\twoheadrightarrow \ZZ^m$ and $\pi \colon G\twoheadrightarrow
F_n$, we have that both indices $[\ZZ^m : H\tau]$ and $[F_n :H\pi]$ must also be finite.
Since we can effectively compute generators for $H\pi$ and for $H\tau$, and we can decide
whether $H\tau \leqslant\fin \ZZ^m$ and $H\pi \leqslant\fin F_n$ hold, we have two
effectively checkable necessary conditions for $H$ to be of finite index in $G$: if either
$[\ZZ^m :H\tau]$ or $[F_n :H\pi]$ is infinite, then so is $[G:H]$.

Nevertheless, these two necessary conditions together are not sufficient to ensure
finiteness of $[G:H]$, as the following easy example shows: take $H=\langle sa, tb\rangle$,
a subgroup of $G=\ZZ^2 \times F_2 = \langle s,t \mid [s,t]\rangle \times \langle a,b \mid
\, \rangle$. It is clear that $H\tau =\ZZ^2$ and $H\pi =F_2$ (so, both indices are 1), but
the index $[\ZZ^2 \times F_2 :H]$ is infinite because no power of $a$ belongs to $H$.

Note that $H\cap \ZZ^m \leqslant H\tau \leqslant \ZZ^m$ and $H\cap F_n \leqslant H\pi
\leqslant F_n$ and, according to Lemma~\ref{lem:index producte directe}, the conditions
really necessary, and sufficient, for $H$ to be of finite index in $G$ are
\begin{equation}\label{cond f.i.}
H\leqslant\fin G \,\, \Leftrightarrow \,\, \begin{cases} H\cap \ZZ^m \leqslant\fin \ZZ^m, \\ H\cap F_n \leqslant\fin H\pi, \mbox{ and } H\pi \leqslant\fin F_n,
\end{cases}
\end{equation}
both stronger than $H\tau \leqslant\fin \ZZ^m$ and $H\pi \leqslant\fin F_n$ respectively
(and none of them satisfied in the example above). This is the main observation which leads
to the following result.

\goodbreak

\begin{thm} \label{th:problema de index finit}
\index{problema de l'\'{\i}ndex finit!decidibilitat i c\`{a}lcul}
The Finite Index Problem for $\ZZ^m \times F_n$ is solvable.
\end{thm}

\begin{proof}
From the given generators for $H$, we start by computing a basis of $H$ (see
Proposition~\ref{prop:bases algorismiques}),
 $$
\{ \mathbf{t}^\mathbf{b_1}, \ldots , \mathbf{t}^\mathbf{b_{m'}},\, \mathbf{t}^\mathbf{a_1} u_1, \ldots
,\mathbf{t}^\mathbf{a_{n'}} u_{n'} \},
 $$
where $0\leqslant m'\leqslant m$, $0\leqslant n'\leqslant p$, $L=\langle \mathbf{b_1},
\ldots , \mathbf{b_{m'}}\rangle \isom \ZZ^{m'}$ with abelian basis~$\{ \mathbf{b_1}, \ldots
, \mathbf{b_{m'}} \}$, $\mathbf{a_1}, \ldots, \mathbf{a_{n'}} \in \ZZ^m$, and $H\pi
=\langle u_1, \ldots, u_{n'} \rangle \isom F_{n'}$ with free basis~$\{ u_1, \ldots,
u_{n'}\}$. As above, let us write $\mathbf{A}$ for the ${n'\times m}$ integral matrix whose
rows are $\mathbf{a_i} \in \ZZ^m$, $i\in[n']$.

Note that $L=\langle \mathbf{b_1}, \ldots, \mathbf{b_{m'}}\rangle \isom H\cap \ZZ^m$ (with
the natural isomorphism $\mathbf{b}\mapsto \mathbf{t^b}$, changing the notation from
additive to multiplicative). Hence, the first necessary condition in~(\ref{cond f.i.}) is
$\rank(L)=m$, i.e.\ $m'=m$. If this is not the case, then $[G:H]=\infty$ and we are done.
So, let us assume $m'=m$ and compute a set of (right) coset representatives for $L$ in
$\ZZ^m$, say $\ZZ^m =\mathbf{c_1}L\sqcup \cdots \sqcup \mathbf{c_r}L$.

Next, check whether $H\pi=\langle u_1, \ldots ,u_{n'}\rangle$ has finite index in $F_n$ (by
computing the core of the Schreier graph of $H\pi$, and checking whether is it complete or
not). If this is not the case, then $[G:H]=\infty$ and we are done as well. So, let us
assume $H\pi \leqslant\fin F_n$, and compute a set of right coset representatives for
$H\pi$ in $F_n$, say $F_n =v_1(H\pi)\sqcup \cdots \sqcup v_s(H\pi)$.

According to~(\ref{cond f.i.}), it only remains to check whether the inclusion $H\cap F_n
\leqslant H\pi$ has finite or infinite index. Call $\rho \colon F_{n'} \twoheadrightarrow
\ZZ^{n'}$ the abstract abelianization map for the free group of rank $n'$ (with free basis
$\{ u_1,\ldots ,u_{n'}\}$), and $A\colon \ZZ^{n'} \to \ZZ^m$ the linear mapping
$\mathbf{v}\mapsto \mathbf{v}\mathbf{A}$ corresponding to right multiplication by the
matrix $\mathbf{A}$. Note that
 $$
H\cap F_n = \{ w\in F_n \mid \mathbf{0}\in \cab_{w,H} \} =\{ w\in F_n \mid \boldsymbol\omega \mathbf{A} \in L
\} \leqslant H\pi,
 $$
where $\boldsymbol\omega =\omega \rho$ is the abelianization of the word $\omega$ which
expresses $w$ in the free basis $\{ u_1,\ldots,u_{n'} \}$ of $H\pi$, i.e.\ $F_n \ni
w=\omega(u_1,\ldots ,u_{n'})$, see Corollary~\ref{cor:propietats complecio abeliana}. Thus,
$H\cap F_n$ is, in terms of the free basis $\{ u_1,\ldots ,u_{n'}\}$, the successive full
preimage of $L$, first by the map $A$ and then by the map $\rho$, namely
$(L)A^{-1}\rho^{-1}$, see the following diagram:
\begin{equation} \label{eq:diagrama index finit}\index{problema de l'\'{\i}ndex finit!diagrama}
\begin{aligned} 
 \xy
 (-21,0)*+{H \pi};
 (-21,-9)*+{H \cap F_n};
 (-21,-5)*+{\rotatebox[origin=c]{90}{$\normaleq$}};
 (0,-5)*+{\rotatebox[origin=c]{90}{$\normaleq$}};
 (23,-5)*+{\rotatebox[origin=c]{90}{$\normaleq$}};
 (41,-5)*+{\rotatebox[origin=c]{90}{$\normaleq$}};
 (-12,-9)*+{\isom};
 (-12,0)*+{\isom};
 (-27.5,0)*+{\geqslant};
 (-33,0)*+{F_n};
 {\ar@{->>}^-{\rho} (0,0)*+++{F_{n'}}; (23,0)*+++{\ZZ^{n'}}};
 {\ar^-{A} (23,0)*++++{}; (41,0)*+{\ZZ^m}};
 {\ar@{|->} (41,-9)*++{L}; (23,-9)*++{(L)A^{-1}}};
 {\ar@{|->} (23,-9)*+++++++{}; (0,-9)*+{(L)A^{-1} \rho^{-1} }};
 \endxy \\[-15pt]
\end{aligned}
\end{equation}
Hence, using Corollary~\ref{lem:index de subgrups a traves de epimorfismes}~(ii), $[H\pi
:H\cap F_n]=[F_{n'}:(L)A^{-1}\rho^{-1}]$ is finite if and only if $[\ZZ^{n'}:(L)A^{-1}]$ is
finite. And this happens if and only if $\rank((L)A^{-1})=n'$. Since
$\rank((L)A^{-1})=\rank((L\cap \im(A))A^{-1})=\rank(L\cap \im(A))+\rank(\ker(A))$, we can
immediately check whether this rank equals $n'$, or not. If this is not the case, then
$[H\pi :H\cap F_n]=[F_{n'}: (L)A^{-1}\rho^{-1}]=[\ZZ^{n'}:(L)A^{-1}]=\infty$ and we are
done. Otherwise, $(L)A^{-1}\leqslant\fin \ZZ^{n'}$ and so, $H\cap F_n \leqslant\fin H\pi$
and $H\leqslant\fin G$.

Finally, suppose $H\leqslant\fin G$ and let us explain how to compute a set of right coset
representatives for $H$ in $G$ (and so, the actual value of the index $[G:H]$). Having
followed the algorithm described above, we have $\ZZ^m =\mathbf{c_1}L\sqcup \cdots \sqcup
\mathbf{c_r}L$ and $F_n =v_1(H\pi)\sqcup \cdots \sqcup v_s(H\pi)$. Furthermore, from the
situation in the previous paragraph, we can compute a set of (right) coset representatives
for $(L)A^{-1}$ in~$\ZZ^{n'}$, which can be easily converted (see Lemma~\ref{lem:index de
subgrups a traves de epimorfismes}~(ii)) into a set of right coset representatives for
$H\cap F_n$ in~$H\pi$, say $H\pi =w_1(H\cap F_n)\sqcup \cdots \sqcup w_t(H\cap F_n)$.

Hence, $F_n =\bigsqcup_{j\in [s]} \bigsqcup_{k\in [t]} v_jw_k (H\cap F_n)$, and $[F_n :
H\cap F_n]=st$. Combining this with $\ZZ^m =\bigsqcup_{i\in [r]}\mathbf{t^{c_i}}(H\cap
\ZZ^m)$, and using the map in the proof of Lemma~\ref{lem:index producte directe}, we get
$G=\ZZ^m \times F_n =\bigcup_{i\in [r]}\bigcup_{j\in [s]}\bigcup_{k\in [t]}
\mathbf{t^{c_i}}v_jw_k H$.

It only remains a cleaning process in the family of $rst$ elements $\{
\mathbf{t^{c_i}}v_jw_k \mid {i\in [r]},\, {j\in [s]},\, {k\in [t]}\}$ to eliminate possible
duplications as representatives of right cosets of $H$ (this can be easily done by several
applications of the membership problem for $H$, see Corollary~\ref{cor:membership}). After
this cleaning process, we get a genuine set of right coset representatives for $H$ in $G$,
and the actual value of~$[G:H]$ (which is at most $rst$).

Finally, inverting all of them we will get a set of left coset representatives for~$H$
in~$G$.
\end{proof}

Regarding the computation of the index $[G:H]$, we remark that the inequality among indices
in Lemma~\ref{lem:index producte directe} may be strict, i.e.\ $[G:H]$ may be strictly less
than~$rst$, as the following example shows.

\begin{exm} \label{exm:contraexemple igualtat indexs}
Let $G=\ZZ^2 \times F_2 =\langle s,t \mid [s,t] \,\rangle \times \langle a,b \mid
\,\rangle$ and consider the (normal) subgroups $H=\langle s, t^2, a, b^2, bab\rangle$ and
$H'=\langle s, t^2, a, b^2, bab, tb\rangle =\langle s, t^2, a, tb\rangle$ of $G$ (with
bases $\{ s, t^2, a, b^2, bab\}$ and $\{ s, t^2, a, tb\}$, respectively). We have $H\cap
\ZZ^2 =H'\cap \ZZ^2 =\langle s , t^2 \rangle \leqslant_{2} \ZZ^2$, and $H\cap F_2 =H'\cap
F_2 =\langle a, b^2, bab\rangle \leqslant_{2} F_2$, but
 $$
[\ZZ^2 \times F_2:H]=4=[\ZZ^2 :H\cap \ZZ^2]\cdot [F_2 :H\cap F_2] ,
 $$
while
 $$
[\ZZ^2 \times F_2:H']=2<4=[\ZZ^2 :H'\cap \ZZ^2]\cdot [F_2 :H'\cap F_2],
 $$
with (right) coset representatives $\{1, b, t, tb\}$ and $\{1, t\}$, respectively. This
shows that both the equality and the strict inequality can occur in Lemma~\ref{lem:index
producte directe}.
\end{exm}

\section{The coset intersection problem and Howson's property} \label{sec:CIP}

Consider the following two related algorithmic problems in an arbitrary group~$G$:

\begin{problem}[\textbf{Subgroup Intersection Problem, $\SIP(G)$}]
Given finitely generated subgroups $H$ and $H'$ of $G$ (by finite sets of generators),
decide whether the intersection $H\cap H'$ is finitely generated and, if so, compute a set
of generators for it.
\end{problem}

\begin{problem}[\textbf{Coset Intersection Problem, $\CIP(G)$}]
Given finitely generated subgroups $H$ and $H'$ of $G$ (by finite sets of generators), and
elements $g,g'\in G$, decide whether the right cosets $gH$ and $g'H'$ intersect trivially
or not; and in the negative case (i.e.\ when $gH\cap g'H'=g''(H\cap H')$), compute such
a~$g''\in G$.
\end{problem}

A group $G$ is said to have the \emph{Howson property} if the intersection of every pair
(and hence every finite family) of finitely generated subgroups ${H,H'\leqslant\fg G}$ is
again finitely generated, ${H\cap H'\leqslant\fg G}$.

It is obvious that $\ZZ^m$ satisfies Howson property, since every subgroup is free-abelian
of rank less than or equal to $m$ (and so, finite). Moreover, $\SIP(\ZZ^m)$ and
$\CIP(\ZZ^m)$ just reduce to solving standard systems of linear equations.

The case of free groups is more interesting. Howson himself established in 1954 that $F_n$
also satisfies the Howson property, see~\cite{howson_intersection_1954}. Since then, there
has been several improvements of this result in the literature, both about shortening the
upper bounds for the rank of the intersection, and about simplifying the arguments used.
The modern point of view is based on the pull-back technique for graphs: one can
algorithmically represent subgroups of $F_n$ by the core of their Schreier graphs, and the
graph corresponding to $H\cap H'$ is the pull-back of the graphs corresponding to $H$ and
$H'$, easily constructible from them. This not only confirms  Howson's property for $F_n$
(namely, the pull-back of finite graphs is finite) but, more importantly, it provides the
algorithmic aspect into the topic by solving  $\SIP(F_n)$. And, more generally, an easy
variation of these arguments using pullbacks also solves  $\CIP(F_n)$, see Proposition~6.1
in~\cite{bogopolski_orbit_2009}.

Baumslag~\cite{baumslag_intersections_1966} established, as a generalization of Howson's
result, the conservation of Howson's property under free products, i.e.\ if $G_1$ and $G_2$
satisfy Howson property then so does $G_1 *G_2$. Despite it could seem against intuition,
the same result fails dramatically when replacing the free product by a direct product. And
one can find an extremely simple counterexample for this, in the family of free-abelian
times free groups; the following observation is folklore (it appears
in~\cite{burns_intersection_1998} attributed to Moldavanski, and as the solution to
exercise~23.8(3) in~\cite{bogopolski_introduction_2008}).

\begin{obs} \label{prop:Fn x Z^n no Howson}
The group $\ZZ^m \times F_n$, for $m\geqslant 1$ and $n\geqslant 2$, does not satisfy the
Howson property.
\end{obs}

\begin{proof}
In $\ZZ \times F_2 =\langle t \mid \, \rangle \times \langle a,b \mid \, \rangle$, consider
the (finitely generated) subgroups $H=\langle a,b \rangle$ and $H'=\langle ta,b \rangle$.
Clearly,
\begin{align*}
H\cap H' & =  \{ w(a,b) \mid w\in F_2\} \cap \{ w(ta,b) \mid w\in F_2\} \\ & = \{ w(a,b) \mid w\in F_2\} \cap \{ t^{|w|_a}w(a,b) \mid w\in F_2\} \\ & = \{ t^0 w(a,b) \mid w\in F_2,\,\, |w|_a=0 \} \\ & = \llangle b\rrangle_{F_2} =\langle a^{-k} b a^{k} ,\, k\in \ZZ \rangle,
\end{align*}
where $|w|_a$ is the total $a$-exponent of $w$ (i.e.\ the first coordinate of the
abelianization $\mathbf{w}\in \ZZ^2$ of $w\in F_2$). It is well known that the normal
closure of $b$ in $F_2$ is not finitely generated, hence $\ZZ \times F_2$ does not satisfy
the Howson property. Since $\ZZ \times F_2$ embeds in $\ZZ^m \times F_n$ for all
$m\geqslant 1$ and $n\geqslant 2$, the group $\ZZ^m \times F_n$ does not have this property
either.
\end{proof}

We remark that the subgroups $H$ and $H'$ in the previous counterexample are both
isomorphic to $F_2$. So, interestingly, the above is a situation where two free groups of
rank 2 have a non-finitely generated (of course, free) intersection. This does not
contradict the Howson property for free groups, but rather indicates that one cannot embed
$H$ and $H'$ simultaneously into a free subgroup of $\ZZ \times F_2$.

In the present section, we shall solve $\SIP(\ZZ^m\times F_n)$ and $\CIP(\ZZ^m\times F_n)$.
The key point is Corollary~\ref{cor:H fg sii Hpi fg}\,: $H\cap H'$ is finitely generated if
and only if $(H\cap H')\pi \leqslant F_n$ is finitely generated. Note that the group $H\pi
\cap H'\pi$ is always finitely generated (by Howson property of $F_n$), but the inclusion
$(H\cap H')\pi \leqslant H\pi \cap H'\pi$ is not (in general) an equality (for example, in
$\ZZ \times F_2 =\langle t \mid \, \rangle \times \langle a,b \mid \, \rangle$, the
subgroups $H=\langle t^2, ta^2 \rangle$ and $H'=\langle t^2, t^{2}a^3\rangle$ satisfy $a^6
\in H\pi \cap H'\pi$ but $a^6 \not\in (H\cap H')\pi$). This opens the possibility for
$(H\cap H')\pi$, and so $H\cap H'$, to be non finitely generated, as is the case in the
example from Observation~\ref{prop:Fn x Z^n no Howson}.

Let us describe in detail the data involved in $\CIP(G)$ for $G=\ZZ^m\times F_n$. By
Proposition~\ref{prop:bases algorismiques}, we can assume that the initial finitely
generated subgroups $H, H'\leqslant G$ are given by respective bases i.e.\ by two sets of
elements
\begin{equation}\label{ee'}
\begin{aligned}
E &=\{ \mathbf{t}^\mathbf{b_1}, \ldots , \mathbf{t}^\mathbf{b_{m_1}}, \mathbf{t}^\mathbf{a_1} u_1,\ldots ,
\mathbf{t}^\mathbf{a_{n_1}} u_{n_1}\}, \\ E' &=\{ \mathbf{t}^\mathbf{b'_1}, \ldots ,\mathbf{t}^\mathbf{b'_{m_2}},
\mathbf{t}^\mathbf{a'_1} u'_1, \ldots ,\mathbf{t}^\mathbf{a'_{n_2}}u'_{n_2}\},
\end{aligned}
\end{equation}
where $\{ u_1, \ldots, u_{n_1} \}$ is a free basis of $H\pi \leqslant F_n$, $\{ u'_1,
\ldots, u'_{n_2} \}$ is a free basis of $H'\pi \leqslant F_n$, $\{ \mathbf{t}^\mathbf{b_1},
\ldots ,\mathbf{t}^\mathbf{b_{m_1}} \}$ is an abelian basis of $H\cap \ZZ^m$, and $\{
\mathbf{t}^\mathbf{b'_1}, \ldots ,\mathbf{t}^\mathbf{b'_{m_2}} \}$ is an abelian basis of
$H'\cap \ZZ^m$. Consider the subgroups $L=\langle \mathbf{b_1}, \ldots,
\mathbf{b_{m_1}}\rangle \leqslant \ZZ^m$ and $L'=\langle \mathbf{b'_1}, \ldots,
\mathbf{b'_{m_2}} \rangle \leqslant \ZZ^m$, and the matrices
 $$
\mathbf{A} = \left( \begin{matrix} \mathbf{a_{1}} \\ \vdots \\ \mathbf{a_{n_1}} \end{matrix}
\right) \in \mathcal{M}_{n_1 \times m}(\ZZ) \quad \text{ and } \quad \mathbf{A'} =\left(
\begin{matrix} \mathbf{a'_{1}} \\ \vdots \\ \mathbf{a'_{n_2}} \end{matrix} \right) \in
\mathcal{M}_{n_2 \times m}(\ZZ).
 $$
We are also given two elements $g=\mathbf{t^{a}}u$ and $g'=\mathbf{t^{a'}}u'$ from $G$, and
have to algorithmically decide whether the intersection $gH\cap g'H'$ is empty or not.

Before start describing the algorithm, note that $H\pi$ is a free group of rank $n_1$.
Since $\{ u_1, \ldots, u_{n_1}\}$ is a free basis of $H\pi$, every element $w\in H\pi$ can
be written in a unique way as a word on the $u_i$'s, say $w=\omega (u_1,\ldots ,u_{n_1})$.
Abelianizing this word, we get the abelianization map $\rho_1 \colon H\pi
\twoheadrightarrow \ZZ^{n_1}$, $w\mapsto \boldsymbol{\omega}$ (not to be confused with the
restriction to $H\pi$ of the ambient abelianization $F_n \twoheadrightarrow \ZZ^n$, which
will have no role in this proof). Similarly, we define the morphism $\rho_2 \colon H'\pi
\twoheadrightarrow \ZZ^{n_2}$.

\goodbreak

With all this data given, note that $gH\cap g'H'$ is empty if and only if its projection to
the free component is empty,
 $$
gH\cap g'H'=\emptyset \,\, \Leftrightarrow \,\, (gH\cap g'H')\pi =\emptyset ;
 $$
so, it will be enough to study this last projection. And, since this projection contains
precisely those elements from $(gH)\pi \cap (g'H')\pi =(u\cdot H\pi)\cap (u'\cdot H'\pi)$
having compatible abelian completions in $gH\cap g'H'$, a direct application of
Lemma~\ref{lem:descripcio d'un subgrup fg en termes d'una base} gives the following result.

\begin{lem}\label{lem:descr proj lliure int cosets Fn x Zm}
With the above notation, the projection $(gH\cap g'H')\pi$ consists precisely on those
elements $v\in (u\cdot H\pi) \cap (u'\cdot H'\pi)$ such that
\begin{equation}\label{eq:condicio lineal inicial}
N_v =\left( \mathbf{a}+\boldsymbol{\omega}\mathbf{A}+L\right) \cap \left(
\mathbf{a'}+\boldsymbol{\omega'}\mathbf{A'}+L' \right) \neq \emptyset,
\end{equation}
where $\boldsymbol{\omega}=w\rho_1$ and $\boldsymbol{\omega'}=w'\rho_2$ are, respectively,
the abelianizations of the abstract words $\omega \in F_{n_1}$ and $\omega'\in F_{n_2}$
expressing $w=u^{-1}v\in H\pi\leqslant F_n$ and $w'=u'^{\,-1}v\in H'\pi\leqslant F_n$ in
terms of the free bases $\{ u_1,\ldots, u_{n_1} \}$ and $\{ u'_1,\ldots, u'_{n_2} \}$
(i.e.\ $u\cdot \omega(u_1,\ldots, u_{n_1})=v=u'\cdot \omega'(u'_1, \ldots, u'_{n_2})$).
That is,
\begin{equation}\label{eq:descr proj lliure int cosets Fn x Zm}
(gH\cap g'H')\pi = \{ v\in (u\cdot H\pi)\cap (u'\cdot H'\pi) \mid N_v \neq \emptyset \, \} \subseteq (u\cdot H\pi) \cap (u'\cdot H'\pi) \tag*{\qed}
\end{equation}
\end{lem}

\begin{thm}\label{thm:cip}
The Coset Intersection Problem for $\ZZ^m \times F_n$ is solvable.
\end{thm}

\begin{proof}
Let $G=\ZZ^m \times F_n$ be a finitely generated free-abelian times free group. Using the
solution to  $\CIP(F_n)$, we start by checking whether $(u\cdot H\pi )\cap (u'\cdot H'\pi)$
is empty or not. In the first case $(gH\cap g'H')\pi$, and so $gH\cap g'H'$, will also be
empty and we are done. Otherwise, we can compute $v_0\in F_n$ such that
\begin{equation}\label{inter}
(u\cdot H \pi )\cap (u'\cdot H'\pi )=v_0 \cdot (H\pi \cap H'\pi),
\end{equation}
compute words $\omega_0 \in F_{n_1}$ and $\omega'_0\in F_{n_2}$ such that $u\cdot
\omega_0(u_1,\ldots, u_{n_1})=v_0 =u'\cdot \omega'_0(u'_1,\ldots, u'_{n_2})$, and compute a
free basis, $\{ v_1, \ldots ,v_{n_3}\}$, for $H\pi \cap H'\pi$ together with expressions of
the $v_i$'s in terms of the free bases for $H\pi$ and $H'\pi$, $v_i =\nu_i(u_1,\ldots,
u_{n_1})=\nu'_i(u'_1,\ldots, u'_{n_2})$, $i\in [n_3 ]$.

Let $\rho_3 \colon H\pi \cap H'\pi \twoheadrightarrow \ZZ^{n_3}$ be the corresponding
abelianization map. Abelianizing the words $\nu_i$ and $\nu'_i$, we can compute the rows of
the matrices $\mathbf{P}$ and $\mathbf{P}'$ (of sizes $n_3\times n_1$ and $n_3\times n_2$,
respectively) describing the abelianizations of the inclusion maps $H\pi
\overset{\iota}{\hookleftarrow} H\pi \cap H'\pi \overset{\iota'}{\hookrightarrow} H'\pi$,
see the central part of the diagram \eqref{xypic:esquema interseccio cosets Fn x Z^m}
below.

By~(\ref{inter}), $u^{-1}v_0 \in H\pi$ and $u'^{-1}v_0\in H'\pi$. So, left translation by
$w_0=u^{-1}v_0$ is a permutation of $H\pi$ (not a homomorphism, unless $w_0=1$), say
$\lambda_{w_0} \colon H\pi \to H\pi$, $x\mapsto w_0x=u^{-1}v_0x$. Analogously, we have the
left translation by $w'_0=u'^{-1}v_0$, say $\lambda_{w_0'} \colon H'\pi \to H'\pi$,
$x\mapsto w_0'x=u'^{-1}v_0x$. We include these translations in  our diagram:
\begin{equation} \label{xypic:esquema interseccio cosets Fn x Z^m}
\begin{aligned}
 \xy
 (0,5)*+{\rotatebox[origin=c]{270}{$\leqslant$}};
 (0,10)*{(H \cap H') \pi};
 (0,0)*+{H \pi \cap H' \pi}; (-25,0)*+{H \pi}; (25,0)*+{H' \pi};
 (-50,0)*+{H \pi}; (50,0)*+{H' \pi};
{\ar@{_(->}_-{\iota} (0,0)*++++++++++{}; (-25,0)*++++{}};
{\ar_-{\lambda_{w_0}} (-25,0)*+++++{}; (-50,0)*++++{}};
{\ar@{^(->}^-{\iota'} (0,0)*++++++++++{}; (25,0)*++++{}};
{\ar^-{\lambda_{w'_0}} (25,0)*+++++{}; (50,0)*++++{}};
 (0,-20)*+{\ZZ^{n_3}}; (-25,-20)*+{\ZZ^{n_1}}; (25,-20)*+{\ZZ^{n_2}};
 (-50,-20)*+{\ZZ^{n_1}}; (50,-20)*+{\ZZ^{n_2}};
 {\ar@{->>}^-{\rho_3} (0,0)*+++{}; (0,-20)*+++{}};
 {\ar@{->>}_-{\rho_1} (-25,0)*+++{}; (-25,-20)*+++{}};
 {\ar@{->>}_-{\rho_1} (-50,0)*+++{}; (-50,-20)*+++{}};
 {\ar@{->>}^-{\rho_2} (25,0)*+++{}; (25,-20)*+++{}};
 {\ar@{->>}^-{\rho_2} (50,0)*+++{}; (50,-20)*+++{}};
 (-12.5,-10)*+{///};(-37.5,-10)*+{///};(12.5,-10)*+{///};(37.5,-10)*+{///};
 (0,-20)*+{\ZZ^{n_3}}; (-25,-20)*+{\ZZ^{n_1}}; (25,-20)*+{\ZZ^{n_2}};
 {\ar_-{\mathbf{P}} (0,-20)*+++++{}; (-25,-20)*++++{}};
 {\ar_-{+\boldsymbol{\omega_0}} (-25,-20)*+++++{}; (-50,-20)*++++{}};
 {\ar^-{\mathbf{P}'} (0,-20)*+++++{}; (25,-20)*++++{}};
 {\ar^-{+\boldsymbol{\omega'_0}} (25,-20)*+++++{}; (50,-20)*++++{}};
 (0,-40)*+{\ZZ^{m}};
 {\ar^-{\mathbf{A}} (-25,-20)*++++{}; (0,-40)*++++{}};
 {\ar_-{\mathbf{A'}} (25,-20)*++++{}; (0,-40)*++++{}};
 {\ar_-{\mathbf{A}} (-50,-20)*++++{}; (0,-40)*++++{}};
 {\ar^-{\mathbf{A'}} (50,-20)*++++{}; (0,-40)*++++{}};
 \endxy
 \end{aligned}
 \end{equation}
where $\boldsymbol{\omega_0}=w_0\rho_1 \in \ZZ^{n_1}$ and
$\boldsymbol{\omega'_0}=w'_0\rho_2 \in \ZZ^{n_2}$ are the abelianizations of $w_0$ and
$w'_0$ with respect to the free bases $\{ u_1, \ldots ,u_{n_1}\}$ and $\{ u'_1, \ldots
,u'_{n_2}\}$, respectively.

Now, for every $v\in (u\cdot H\pi)\cap (u'\cdot H'\pi)$, using Lemma~\ref{lem:descr proj
lliure int cosets Fn x Zm} and the commutativity of the upper part of the above diagram, we
have

\begin{align*}
 N_v & =  \left( \mathbf{a}+(u^{-1}v)\rho_1\mathbf{A}+L\right) \cap \left( \mathbf{a'}
 +(u'^{-1}v) \rho_2 \mathbf{A'} +L' \right) \\
& =  \left( \mathbf{a}+(v_0^{-1}v)\iota \lambda_{w_0}\rho_1\mathbf{A}+L\right) \cap
\left( \mathbf{a'}+(v_0^{-1}v)\iota' \lambda_{w'_0}\rho_2 \mathbf{A'} +L' \right) \\
 & =  \left( \mathbf{a}
+(\boldsymbol{\omega_0}+(v_0^{-1}v)\rho_3\mathbf{P})\mathbf{A}+L\right) \cap \left(
\mathbf{a'}+ (\boldsymbol{\omega'_0}+(v_0^{-1}v) \rho_3\mathbf{P'})\mathbf{A'} +L' \right) \\
 & =  \left( \mathbf{a}+\boldsymbol{\omega_0} \mathbf{A} +(v_0^{-1}v)\rho_3
\mathbf{PA}+L\right) \cap \left( \mathbf{a'}+\boldsymbol{\omega'_0} \mathbf{A'}+(v_0^{-1}v)
\rho_3\mathbf{P'A'} +L' \right).
\end{align*}

With this expression, we can characterize, in a computable way, which elements from
$(u\cdot H\pi )\cap (u'\cdot H'\pi )$ do belong to $(gH\cap g'H')\pi$:

\begin{lem}\label{th:projeccio interseccio de cosets de Fn x Zm}
With the current notation we have
\begin{equation}\label{eq:calcul projeccio interseccio cosets}
(gH\cap g'H')\pi =M\rho_3^{-1}\lambda_{v_0} \subseteq (u\cdot H\pi )\cap (u'\cdot H'\pi ),
\end{equation}
where $M\subseteq \ZZ^{n_3}$ is the preimage by the linear mapping $\mathbf{PA-P'A'}\colon
\ZZ^{n_3} \to \ZZ^m$ of the linear variety
\begin{equation}\label{eq:variedad lineal de Z^m}
N=\mathbf{a'-a}+\boldsymbol{\omega'_0} \mathbf{A'}-\boldsymbol{\omega_0} \mathbf{A}+(L+L')
\subseteq \ZZ^m.
\end{equation}
\end{lem}

\begin{proof}
By Lemma~\ref{lem:descr proj lliure int cosets Fn x Zm}, an element $v\in (u\cdot H\pi)
\cap (u'\cdot H'\pi)$ belongs to $(gH\cap g'H')\pi$ if and only if $N_v\neq \emptyset$.
That is, if and only if the vector $\mathbf{x}=(v_0^{-1}v)\rho_3 \in \ZZ^{n_3}$ satisfies
that the two varieties $\mathbf{a}+\boldsymbol{\omega_0} \mathbf{A}+\mathbf{xPA}+L$ and
$\mathbf{a'}+\boldsymbol{\omega'_0} \mathbf{A'}+\mathbf{xP'A'} +L'$ do intersect. But this
happens if and only if the vector
 $$
\big( \mathbf{a}+\boldsymbol{\omega_0} \mathbf{A}+\mathbf{xPA} \big) - \big(\mathbf{a'}+
\boldsymbol{\omega'_0} \mathbf{A'}+\mathbf{xP'A'} \big) =\mathbf{a-a'}+\boldsymbol{\omega_0}
\mathbf{A} -\boldsymbol{\omega'_0} \mathbf{A'}+\mathbf{x(PA-P'A')}
 $$
belongs to $L+L'$. That is, if and only if $\mathbf{x(PA-P'A')}$ belongs to $N$. Hence, $v$
belongs to $(gH\cap g'H')\pi$ if and only if $\mathbf{x}=(v_0^{-1}v)\rho_3 \in M$, i.e.\ if
and only if $v\in M\rho_3^{-1}\lambda_{v_0}$.
\end{proof}

With all the data already computed, we explicitly have the variety $N$ and, using standard
linear algebra, we can compute $M$ (which could be empty, because $N$ may possibly be
disjoint with the image of $\mathbf{PA-P'A'}$). In this situation, the algorithmic decision
on whether $gH\cap g'H'$ is empty or not is straightforward.

\begin{lem}\label{cor:caracteritzacio interseccio cosets buida}
With the current notation, and assuming that $(u\cdot H \pi) \cap (u'\cdot H' \pi) \neq
\emptyset$, the following are equivalent:
\begin{itemize}
\item[\emph{(a)}] $gH\cap g'H'=\emptyset$,
\item[\emph{(b)}] $(gH\cap g'H')\pi =\emptyset$,
\item[\emph{(c)}] $M\rho_3^{-1}=\emptyset$,
\item[\emph{(d)}] $M=\emptyset$,
\item[\emph{(e)}] $N\cap \im(\mathbf{PA-P'A'})=\emptyset$. \qed
\end{itemize}
\end{lem}

If $gH\cap g'H'=\emptyset$, we are done. Otherwise, $N\cap
\im(\mathbf{PA-P'A'})\neq\emptyset$ and we can compute a vector $\mathbf{x}\in \ZZ^{n_3}$
such that $\mathbf{x}(\mathbf{PA-P'A'})\in N$. Take now any preimage of $\mathbf{x}$ by
$\rho_3$, for example $v_1^{x_1}\cdots\, v_{n_3}^{x_{n_3}}$ if $\mathbf{x}=(x_1,\ldots
,x_{n_3})$, and by (\ref{eq:calcul projeccio interseccio cosets}),
$u''=v_0v_1^{x_1}\cdots\, v_{n_3}^{x_{n_3}}\in (gH\cap g'H')\pi$.

It only remains to find $\mathbf{a''}\in \ZZ^m$ such that $g''=\mathbf{t^{a''}}u'' \in
gH\cap g'H'$. To do this, observe that $u'' \in (gH\cap g'H')\pi$ implies the existence of
a vector $\mathbf{a''}$ such that $\mathbf{t^{a''}}u''\in \mathbf{t^a}uH \cap
\mathbf{t^{a'}}u'H'$, i.e.\  such that $\mathbf{t^{a''-a}}u^{-1}u''\in H$ and
$\mathbf{t^{a''-a'}}u'^{-1}u''\in H'$. In other words, there exists a vector
$\mathbf{a''}\in \ZZ^m$ such that $\mathbf{a''-a}\in \cab_{u^{-1}u'',H}$ and
$\mathbf{a''-a'}\in \cab_{u'^{-1}u'',H'}$. That is, the affine varieties
$\mathbf{a}+\cab_{u^{-1}u'',H}$ and $\mathbf{a'}+\cab_{u'^{-1}u'',H'}$ do intersect. By
Corollary~\ref{cor:propietats complecio abeliana}, we can compute equations for these two
varieties, and compute a vector in its intersection. This is the $\mathbf{a''}\in \ZZ^m$ we
are looking for.
\end{proof}

The above argument applied to the case where $g=g'=1$ is giving us valuable information
about the subgroup intersection $H\cap H'$; this will allow us to solve  $\SIP(\ZZ^m \times
F_n)$ as well. Note that, in this case, $\mathbf{a}=\mathbf{a'}=\mathbf{0}$, $u=u'=1$ and
so, $v_0=1$, $w_0=w'_0=1$, and $\boldsymbol{\omega_0}=\boldsymbol{\omega'_0}=\mathbf{0}$.
\goodbreak
\begin{thm}\label{thm:ip}
The Subgroup Intersection Problem for $\ZZ^m \times F_n$ is solvable.
\end{thm}

\begin{proof}
Let $G=\ZZ^m \times F_n$ be a finitely generated free-abelian times free group. As in the
proof of Theorem~\ref{thm:cip}, we can assume that the initial finitely generated subgroups
$H, H'\leqslant G$ are given by respective bases, i.e.\ by two sets of elements like
in~(\ref{ee'}), $E=\{ \mathbf{t}^\mathbf{b_1}, \ldots ,\mathbf{t}^\mathbf{b_{m_1}},
\mathbf{t}^\mathbf{a_1} u_1,\ldots , \mathbf{t}^\mathbf{a_{n_1}} u_{n_1}\}$ and $E'=\{
\mathbf{t}^\mathbf{b'_1}, \ldots ,\mathbf{t}^\mathbf{b'_{m_2}}, \mathbf{t}^\mathbf{a'_1}
u'_1, \ldots ,\mathbf{t}^\mathbf{a'_{n_2}}u'_{n_2}\}$. Consider the subgroups $L,
L'\leqslant \ZZ^m$ and the matrices $\mathbf{A}\in \mathcal{M}_{n_1 \times m}(\ZZ)$ and
$\mathbf{A'}\in \mathcal{M}_{n_2 \times m}(\ZZ)$ as above. We shall algorithmically decide
whether the intersection $H\cap H'$ is finitely generated or not and, in the affirmative
case, shall compute a basis for $H\cap H'$.

Let us apply the algorithm from the proof of Theorem~\ref{thm:cip} to the cosets $1\cdot H$
and $1\cdot H'$; that is, take $g=g'=1$, i.e.\ $u=u'=1$ and
$\mathbf{a}=\mathbf{a'}=\mathbf{0}$. Of course, $H\cap H'$ is not empty, and $v_0=1$ serves
as an element in the intersection, $v_0\in H\cap H'$. With this choice, the algorithm works
with $w_0=w'_0=1$ and $\boldsymbol{\omega_0}=\boldsymbol{\omega'_0}=\mathbf{0}$ (so, we can
forget the two translation parts in diagram~\eqref{xypic:esquema interseccio cosets Fn x
Z^m}). Lemma~\ref{th:projeccio interseccio de cosets de Fn x Zm} tells us that $(H\cap
H')\pi =M\rho_3^{-1}\leqslant H\pi \cap H'\pi$, where $M$ is the preimage by the linear
mapping $\mathbf{PA-P'A'}\colon \ZZ^{n_3}\to \ZZ^{n_1}$ of the subspace $N=L+L'\leqslant
\ZZ^m$. In this situation, the following lemma decides when is $H\cap H'$ finitely
generated and when is not:

\begin{lem}\label{int-fg}
With the current notation, the following are equivalent:
\begin{itemize}
\item[\emph{(a)}] $H\cap H'$ is finitely generated,
\item[\emph{(b)}] $(H\cap H')\pi$ is finitely generated,
\item[\emph{(c)}] $M\rho_3^{-1}$ is either trivial or of finite index in $H\pi \cap
    H'\pi$,
\item[\emph{(d)}] either $n_3=1$ and $M=\{ \mathbf{0}\}$, or $M$ is of finite index in
    $\ZZ^{n_3}$,
\item[\emph{(e)}] either $n_3=1$ and $M=\{ \mathbf{0}\}$, or $\rank(M)=n_3$.
\end{itemize}
\end{lem}

\begin{proof}
(a) $\Leftrightarrow$ (b) is in Corollary~\ref{cor:H fg sii Hpi fg}. (b) $\Leftrightarrow$
(c) comes from the well known fact (see, for example, \cite{lyndon_combinatorial_2001}
pags. 16-18) that, in the finitely generated free group $H\pi \cap H'\pi$, the subgroup
$(H\cap H')\pi =M\rho_3^{-1}$ is normal and so, finitely generated if and only if it is
either trivial or of finite index. But, by lemma~\ref{lem:index de subgrups a traves de
epimorfismes}~(ii), the index $[H\pi \cap H'\pi : M\rho_3^{-1}]$ is finite if and only if
$[\ZZ^{n_3} : M]$ is finite; this gives (c) $\Leftrightarrow$ (d). The last equivalence is
a basic fact in linear algebra.
\end{proof}

We have computed $n_3$ and an abelian basis for $M$. If $n_3=0$ we immediately deduce that
$H\cap H'$ is finitely generated. If $n_3=1$ and $M=\{ \mathbf{0}\}$ we also deduce that
$H\cap H'$ is finitely generated. Otherwise, we check whether $\rank(M)$ equals $n_3$; if
this is the case then again $H\cap H'$ is finitely generated; if not, $H\cap H'$ is
infinitely generated.

It only remains to algorithmically compute a basis for $H\cap H'$, in case it is finitely
generated. We know from~\eqref{eq:factoritzacio subgrup} that
 $$
H\cap H'=\bigl((H\cap H')\cap \ZZ^m \bigr)\, \times \, (H\cap H')\pi \alpha,
 $$
where $\alpha$ is any splitting for $\pi_{\mid H\cap H'} \colon H\cap H'\twoheadrightarrow
(H\cap H')\pi$; then we can easily get a basis of $H\cap H'$ by putting together a basis of
each part. The strategy will be the following: first, we compute an abelian basis for
 $$
(H\cap H')\cap \ZZ^m =(H\cap \ZZ^m)\cap (H'\cap \ZZ^m)=L\cap L'
 $$
by just solving a system of linear equations. Second, we shall compute a free basis for
$(H\cap H')\pi$. And finally, we will construct an explicit splitting $\alpha$ and will use
it to get a free basis for $(H\cap H')\pi \alpha$. Putting together these two parts, we
shall be done.

To compute a free basis for $(H\cap H')\pi$ note that, if $n_3=0$, or $n_3=1$ and $M=\{
\mathbf{0}\}$, then $(H\cap H')\pi =1$ and there is nothing to do. In the remaining case,
$\rank(M)=n_3\geqslant 1$, $M\rho_3^{-1}=(H\cap H')\pi$ has finite index in $H\pi \cap
H'\pi$, and so it is finitely generated. We give two alternative options to compute a free
basis for it.

The subgroup $M$ has finite index in $\ZZ^{n_3}$, and we can compute a system of coset
representatives of $\ZZ^{n_3}$ modulo $M$,
 $$
\ZZ^{n_3}=M\mathbf{c_1}\sqcup \cdots \sqcup M\mathbf{c_d}
 $$
(see the beginning of Section~\ref{fi}). Now, being $\rho_3$ onto, and according to
Lemma~\ref{lem:index de subgrups a traves de epimorfismes}~(b), we can transfer the
previous partition via $\rho_3$ to obtain a system of right coset representatives of $H\pi
\cap H'\pi$ modulo $M\rho_3^{-1}$:
\begin{equation}\label{eq:classes modul M rho_3^(-1)}
H\pi \cap H'\pi =(M\rho_3^{-1})z_1\sqcup \cdots \sqcup (M\rho_3^{-1})z_d,
\end{equation}
where we can take, for example, $z_i =v_1^{c_{i,1}} v_2^{c_{i,2}}\cdots\,
v_{n_3}^{c_{i,n_3}}\in H\pi \cap H'\pi$, for each vector $\mathbf{c_i}=(c_{i,1},
c_{i,2},\ldots, c_{i,n_3}) \in \ZZ^{n_3}$, $i\in [d]$. Now let us construct the core of the
Schreier graph for $M\rho_3^{-1}=(H\cap H')\pi$ (with respect to $\{v_1, \ldots
,v_{n_3}\}$, a free basis for $H\pi \cap H'\pi$), $\mathcal{S}(M\rho_3^{-1})$, in the
following way: consider the graph with the cosets of~\eqref{eq:classes modul M rho_3^(-1)}
as vertices, and with no edge. Then, for every vertex $(M\rho_3^{-1})z_i$ and every letter
$v_j$, add an edge labeled $v_j$ from $(M\rho_3^{-1})z_i$ to $(M \rho_3^{-1})z_i v_j$,
algorithmically identified among the available vertices by repeatedly using the
\emph{membership problem} for $M\rho_3^{-1}$ (note that we can do this by abelianizing the
candidate and checking the defining equations for $M$). Once we have run over all $i,j$, we
shall get the full graph $\mathcal{S}(M\rho_3^{-1})$, from which we can easily obtain a
free basis for~$(H\cap H')\pi$ in terms of~$\{v_1, \ldots ,v_{n_3}\}$.

Alternatively, let $\{ \mathbf{m_1}, \ldots ,\mathbf{m_{n_3}}\}$ be an abelian basis for
$M$ (which we already have from the previous construction), say $\mathbf{m_i}=(m_{i,1},
m_{i,2},\ldots, m_{i,n_3}) \in \ZZ^{n_3}$, $i=1,\ldots ,n_3$, and consider the elements
$x_i =v_1^{m_{i,1}} v_2^{m_{i,2}}\cdots\, v_{n_3}^{m_{i,n_3}}\in H\pi \cap H'\pi$. It is
clear that $M\rho_3^{-1}$ is the subgroup of $H\pi \cap H'\pi$ generated by $x_1, \ldots
,x_{n_3}$ and all the infinitely many commutators from elements in $H\pi \cap H'\pi$. But
$M\rho_3^{-1}$ is finitely generated so, finitely many of those commutators will be enough.
Enumerate all of them, $y_1, y_2, \ldots$ and keep computing the core $\mathcal{S}_j$ of
the Schreier graph for the subgroup $\langle x_1, \ldots , x_{n_3}, y_1, \ldots
,y_j\rangle$ for increasing $j$'s until obtaining a complete graph with $d$ vertices (i.e.\
until reaching a subgroup of index $d$). When this happens, we shall have computed the core
of the Schreier graph for $M\rho_3^{-1}=(H\cap H')\pi$ (with respect to $\{v_1, \ldots
,v_{n_3}\}$, a free basis of $H\pi \cap H'\pi$), from which we can easily find a free basis
for $(H\cap H')\pi$, in terms of~$\{v_1, \ldots ,v_{n_3}\}$.

Finally, it remains to compute an explicit splitting $\alpha$ for $\pi_{\mid H\cap H'}
\colon H\cap H'\twoheadrightarrow (H\cap H')\pi$. We have a free basis $\{ z_1,\ldots ,z_d
\}$ for $(H\cap H')\pi$, in terms of $\{ v_1,\ldots ,v_{n_3}\}$; so, using the expressions
$v_i =\nu_i (u_1,\ldots, u_{n_1})$ that we have from the beginning of the proof, we can get
expressions $z_i =\eta_i (u_1,\ldots ,u_{n_1})$. From here, $\eta_i
(\mathbf{t^{a_1}}u_1,\ldots ,\mathbf{t^{a_{n_1}}}u_{n_1}) =\mathbf{t^{e_i}} z_i \in H$ and
projects to $z_i$, so $\mathcal{C}_{z_i,\, H}=\mathbf{e_i}+L$ (see
Corollary~\ref{cor:propietats complecio abeliana}), $i\in [d]$. Similarly, we can get
vectors $\mathbf{e'_i}\in \mathbb{Z}^m$ such that $\mathcal{C}_{z_i,\,
H'}=\mathbf{e'_i}+L'$. Since, by construction, $\mathcal{C}_{z_i, H\cap H'}
=\mathcal{C}_{z_i, H}\cap \mathcal{C}_{z_i, H'}$ is a non-empty affine variety in
$\mathbb{Z}^m$ with direction $L\cap L'$, we can compute vectors $\mathbf{e''_i}\in
\mathbb{Z}^m$ on it by just solving the corresponding systems of linear equations, $i\in
[d]$. Now, $z_i \mapsto \mathbf{t^{e''_i}}z_i$ is the desired splitting $H\cap
H'\stackrel{\alpha}{\leftarrow} (H\cap H')\pi$, and $\{ \mathbf{t^{e''_1}}z_1,\, \ldots
,\mathbf{t^{e''_d}}z_d \}$ is the free basis for $(H\cap H')\pi\alpha$ we were looking for.

As mentioned above, putting together this free basis with the abelian basis we already have
for $L\cap L'$, we get a basis for $H\cap H'$, concluding the proof.
\end{proof}

\begin{cor} \label{cor:interseccio subgrups lliures no abelians de rang finit}
Let $H,H'$ be two free non-abelian subgroups of finite rank in $\ZZ^m \times F_n $. With
the previous notation, the intersection $H\cap H'$ is finitely generated if and only if
either $H\cap H'=1$, or~${\mathbf{PA}=\mathbf{P'A'}}$.
\end{cor}

\begin{proof}
Under the conditions of the statement, we have $L=L'=\{ \mathbf{0}\}$. Hence, $N=L+L'=\{
\mathbf{0}\}$ and its preimage by $\mathbf{PA-P'A'}$ is $M=\ker(\mathbf{PA-P'A'})\leqslant
\ZZ^{n_3}$. Now, by Lemma~\ref{int-fg}, $H\cap H'$ is finitely generated if and only if
either $(H\cap H')\pi =M\rho_3^{-1}=1$, or $n_3 -\rank
(Im(\mathbf{PA-P'A'}))=\rank(M)=n_3$; that is, if and only if either $(H\cap H')\pi =1$, or
$\mathbf{PA}=\mathbf{P'A'}$. But, since $L=L'=\{ \mathbf{0}\}$, $(H\cap H')\pi =1$ if and
only if $H\cap H'=1$.
\end{proof}

\goodbreak

We consider the following two examples to illustrate the preceding algorithm.

\begin{exm}
Let us analyze again the example given in the proof of Observation~\ref{prop:Fn x Z^n no
Howson}, under the light of the previous corollary. We considered in $\ZZ \times
F_2=\langle t \mid \, \rangle \times \langle a,b \mid \, \rangle$ the subgroups $H=\langle
a,b \rangle$ and $H'=\langle ta,b\rangle$, both free non-abelian of rank 2. It is clear
that $\mathbf{A}=\left(\begin{smallmatrix} 0 \\ 0 \end{smallmatrix}\right)$ and
$\mathbf{A'}=\left(\begin{smallmatrix} 1 \\ 0 \end{smallmatrix}\right)$, while $H\pi =H'\pi
=H\pi \cap H'\pi =F_2$; in particular, $n_3=2$ and $H\cap H'\neq 1$. In these
circumstances, both inclusions $H\pi \overset{}{\hookleftarrow} H\pi \cap H'\pi
\overset{}{\hookrightarrow} H'\pi$ are the identity maps, so
$\mathbf{P}=\mathbf{P'}=\mathbf{1}$ is the $2\times 2$ identity matrix and hence,
$\mathbf{PA}= \left(\begin{smallmatrix} 0 \\ 0 \end{smallmatrix}\right) \neq
\left(\begin{smallmatrix} 1 \\ 0 \end{smallmatrix}\right) =\mathbf{P'A'}$. According to
Corollary~\ref{cor:interseccio subgrups lliures no abelians de rang finit}, this means that
$H\cap H'$ is not finitely generated, as we had seen before.
\end{exm}

\begin{exm}
Consider two finitely generated subgroups $H, H'\leqslant F_n \leqslant \ZZ^m \times F_n$.
In this case we have $\mathbf{A}=(\mathbf{0})\in \mathcal{M}_{n_1,m}$ and
$\mathbf{A'}=(\mathbf{0})\in \mathcal{M}_{n_2,m}$ and so,
$\mathbf{PA}=(\mathbf{0})=\mathbf{P'A'}$. Thus, Corollary~\ref{cor:interseccio subgrups
lliures no abelians de rang finit} just corroborates Howson's property for finitely
generated free groups.
\end{exm}

\goodbreak

To finish this section, we present an application of Theorem~\ref{thm:ip} to a nice
geometric problem. In the very recent paper~\cite{sahattchieve_quasiconvex_2011}, J. Sahattchieve studies
quasi-convexity of subgroups of $\ZZ^m \times F_n$ with respect to the natural
component-wise action of $\ZZ^m \times F_n$ on the product space, $\mathbb{R}^m \times
T_n$, of the $m$-dimensional euclidean space and the regular $(2n)$-valent infinite tree
$T_n$: a subgroup $H\leqslant \ZZ^m \times F_n$ is \emph{quasi-convex} if the orbit $Hp$ of
some (and hence every) point $p\in \mathbb{R}^m \times T_n$ is a quasi-convex subset of
$\mathbb{R}^m \times T_n$ (see~\cite{sahattchieve_quasiconvex_2011} for more details). One of the results obtained is
the following characterization:

\begin{thm}[Sahattchieve]\label{thmSahattchieve}
Let $H$ be a subgroup of $\ZZ^m \times F_n$. Then, $H$ is quasi-convex if and only if $H$
is either cyclic or virtually of the form $A\times B$, for some $A\leqslant \ZZ^m$ and
$B\leqslant F_n$ being finitely generated. (In particular, quasi-convex subgroups are
finitely generated.)
\end{thm}

Combining this with our Theorem~\ref{thm:ip}, we can easily establish an algorithm to
decide whether a given finitely generated subgroup of $\ZZ^m \times F_n$ is quasi-convex or
not (with respect to the above mentioned action).

\begin{cor}
There is an algorithm which, given a finite list $w_1,\ldots,w_s$ of elements in $\ZZ^m
\times F_n$, decides whether the subgroup $H=\langle w_1,\ldots,w_s \rangle$ is
quasi-convex or not.
\end{cor}

\begin{proof}
First, apply Proposition~\ref{prop:bases algorismiques} to compute a basis for $H$. If it
contains only one element, then $H$ is cyclic and we are done.

Otherwise ($H$ is not cyclic) we can easily compute a free-abelian basis and a free basis for the respective projections $H\tau \leqslant \ZZ^m$ and $H\pi\leqslant F_n$. From the basis for $H$ we can immediately extract a free-abelian basis for $\ZZ^m \cap H=H\tau \cap H$. And, using Theorem~\ref{thm:ip}, we can decide whether $F_n \cap H=H\pi \cap H$ is finitely generated or not and, in the affirmative case, compute a free basis for it.
Finally, we can decide whether $H\tau \cap H \leqslant\fin H\tau$ and $H\pi \cap
H\leqslant\fin H\pi$ hold or not (applying the well known solutions to $\FIP(\mathbb{Z}^m)$
and $\FIP(F_{n'})$ or, alternatively, using the more general Theorem~\ref{th:problema de
index finit} above); note that if we detected that $H\pi \cap H$ is infinitely generated
then it must automatically be of infinite index in $H\pi$ (which, of course, is finitely
generated).

Now we claim that $H$ is quasi-convex if and only if $H\tau \cap H \leqslant\fin H\tau$ and
$H\pi \cap H\leqslant\fin H\pi$; this will conclude the proof.

For the implication to the right (and applying Theorem~\ref{thmSahattchieve}), assume that
$A\times B\leqslant\fin H$ for some $A\leqslant \ZZ^m$ and $B\leqslant F_n$ being finitely
generated. Applying $\tau$ and $\pi$ we get $A\leqslant\fin H\tau$ and $B\leqslant\fin
H\pi$, respectively (see Lemma~\ref{lem:index de subgrups a traves de epimorfismes}~(i)).
But $A\leqslant H\tau \cap H\leqslant H\tau$ and $B\leqslant H\pi \cap H\leqslant H\pi$
hence, $H\tau \cap H \leqslant\fin H\tau$ and $H\pi \cap H\leqslant\fin H\pi$.

For the implication to the left, assume $H\tau \cap H \leqslant\fin H\tau$ and $H\pi \cap
H\leqslant\fin H\pi$ (and, in particular, $H\pi \cap H$ finitely generated). Take $A=H\tau
\cap H\leqslant\fin H\tau \leqslant \ZZ^m$ and $B=H\pi \cap H\leqslant\fin H\pi \leqslant
F_n$, and we get $A\times B \leqslant\fin H\tau \times H\pi$ (see Lemma~\ref{lem:index
producte directe}). But $H$ is in between, $A\times B\leqslant H\leqslant H\tau \times
H\pi$, hence $A\times B\leqslant\fin H$ and, by Theorem~\ref{thmSahattchieve}, $H$ is
quasi-convex.
\end{proof}

\section{Endomorphisms} \label{sec:morphisms}

In this section we will study the endomorphisms of a finitely generated free-abelian times
free group $G=\ZZ^m \times F_n$ (with the notation from presentation~(\ref{eq:pres F_n x
Z^m})). Without loss of generality, we assume $n\neq 1$.

To clarify notation, we shall use lowercase Greek letters to denote endomorphisms of $F_n$,
and uppercase Greek letters to denote endomorphisms of $G=\ZZ^m \times F_n$. The following
proposition gives a description of how all endomorphisms of $G$ look like.

\begin{prop} \label{prop:classificacio endos}
Let $G=\ZZ^m \times F_n$ with $n\neq 1$. The following is a complete list of all
endomorphisms
 of $G$:
\begin{itemize}\label{eq:expressio endosI}
\item[\rm{\textbf{(I)}}] $\Psi_{\phi,\mathbf{Q,P}}=\mathbf{t^a}u\mapsto \mathbf{t^{aQ+uP}}\,
    u\phi$, where $\phi \in \End(F_n)$, $\mathbf{Q}\in \mathcal{M}_{m}(\ZZ)$, and
    $\mathbf{P}\in \mathcal{M}_{n\times m}(\ZZ)$.
\item[\rm{\textbf{(II)}}] $\Psi_{z,\mathbf{l,h,Q,P}}=\mathbf{t^a}u\mapsto
    \mathbf{t^{aQ+uP}}z^{\mathbf{a}\mathbf{l}\tr +\mathbf{u}\mathbf{h}\tr}$, where $1\neq
    z\in F_n$ is not a proper power, $\mathbf{Q}\in \mathcal{M}_{m}(\ZZ)$, $\mathbf{P}\in
    \mathcal{M}_{n\times m}(\ZZ)$, $\mathbf{0}\neq \mathbf{l}\in \ZZ^m $, and
    $\mathbf{h}\in \ZZ^n$.
\end{itemize}
(In both cases, $\mathbf{u}\in \ZZ^n$ denotes the abelianization of the word $u\in F_n$.)
\end{prop}
\begin{proof}
It is straightforward to check that all maps of types (I) and (II) are, in fact,
endomorphisms of $G$.

To see that this is the complete list of all of them, let $\Psi \colon G\to G$ be an
arbitrary endomorphism of $G$. Looking at the normal form of the images of the $x_i$'s and
$t_j$'s, we have
\begin{equation}\label{eq:assignacio generica}
\Psi \colon \left\{ \begin{array}{rcl}
x_i \!\! &\longmapsto &\!\! \mathbf{t^{p_i}}w_i \\
t_j \!\! & \longmapsto &\!\! \mathbf{t^{q_j}}
z_j,
\end{array} \right.
\end{equation}
where $\mathbf{p_i},\mathbf{q_j} \in \ZZ^m$ and $w_i,\, z_j\in F_n$, $i\in [n]$, $j\in
[m]$. Let us distinguish two cases.

\medskip

\emph{Case 1: $z_j=1$ for all $j\in [m]$}. Denoting $\phi$ the endomorphism of $F_n$ given
by $x_i \mapsto w_i$, and $\mathbf{P}$ and $\mathbf{Q}$ the following integral matrices (of
sizes $n\times m$ and $m\times m$, respectively)
 $$
\mathbf{P}=\left( \begin{array}{ccc} p_{11} & \cdots & p_{1m} \\ \vdots & \ddots & \vdots \\
p_{n1} & \cdots & p_{nm} \end{array} \right) =\left( \begin{array}{c} \mathbf{p_{1}} \\
\vdots \\ \mathbf{p_{n}} \end{array} \right) \text{\quad and \quad} \mathbf{Q}=\left(
\begin{array}{ccc} q_{11} & \cdots & q_{1m} \\ \vdots & \ddots & \vdots \\ q_{m1} & \cdots &
q_{mm} \end{array} \right) =\left( \begin{array}{c} \mathbf{q_{1}} \\ \vdots \\
\mathbf{q_{m}} \end{array} \right),
 $$
we can write
 $$
\Psi \colon \left\{ \begin{array}{rcl}
u \!\!& \longmapsto &\!\! \mathbf{t^{uP}} u \phi \\
 \mathbf{t^a} \!\!& \longmapsto &\!\! \mathbf{t^{aQ}},
 \end{array} \right.
 $$
where $u\in F_n$ and $\mathbf{a}\in \ZZ^m$. So, $(\mathbf{t^a}u)\Psi =\mathbf{t^{aQ+uP}}
u\phi$ and $\Psi$ equals $\Psi_{\phi,\mathbf{Q,P}}$ from \ti.

\medskip

\emph{Case 2: $z_k\neq 1$ for some $k\in [m]$}. For $\Psi$ to be well defined,
$\mathbf{t^{p_i}}w_i$ and $\mathbf{t^{q_j}} z_j$ must all commute with $\mathbf{t^{q_k}}
z_k$, and so $w_i$ and $z_j$ with $z_k \neq 1$, for all $i\in [n]$ and $j\in [m]$. This
means that $w_i =z^{h_i}$, $z_j =z^{\, l_j}$ for some integers $h_i,\, l_j \in \ZZ$, $i\in
[n]$, $j\in [m]$, with $l_k\neq 0$, and some $z\in F_n$ not being a proper power. Hence,
$(\mathbf{t^a}u)\Psi =(\mathbf{t^a}\Psi )(u\Psi
)=(\mathbf{t^{aQ}}z^{\mathbf{a}\mathbf{l}\tr})(\mathbf{t^{uP}}z^{\mathbf{u}\mathbf{h}\tr})
=\mathbf{t^{aQ+uP}}z^{\mathbf{a}\mathbf{l}\tr +\mathbf{u}\mathbf{h}\tr}$ and $\Psi$ equals
$\Psi_{z,\mathbf{l,h,Q,P}}$ from \tii.

This completes the proof.
\end{proof}

Note that if $n=0$ then \ti\ and \tii\ endomorphisms do coincide. Otherwise,
\tii\ endomorphisms will be seen to be neither injective nor surjective. The following
proposition gives a quite natural characterization of which endomorphisms of \ti\ are
injective, and which are surjective. It is important to note that the matrix $\mathbf{P}$
plays absolutely no role in this matter.

\begin{prop}\label{prop:caract mono}
Let $\Psi$ be an endomorphism of $G=\ZZ^m \times F_n$, with $n\geqslant 2$. Then,
\begin{itemize}
\item[\emph{(i)}] $\Psi$ is a monomorphism if and only if it is of \ti, $\Psi
    =\Psi_{\phi, \mathbf{Q}, \mathbf{P}}$, with $\phi$ a monomorphism of $F_n$, and $\det
    (\mathbf{Q})\neq 0$,
\item[\emph{(ii)}] $\Psi$ is an epimorphism if and only if it is of \ti, $\Psi
    =\Psi_{\phi, \mathbf{Q}, \mathbf{P}}$, with $\phi$ an epimorphism of $F_n$, and $\det
    (\mathbf{Q})=\pm 1$.
\item[\emph{(iii)}] $\Psi$ is an automorphism if and only if it is of \ti, $\Psi
    =\Psi_{\phi, \mathbf{Q}, \mathbf{P}}$, with $\phi \in \Aut(F_n)$ and $\mathbf{Q} \in
    GL_m(\ZZ)$; in this case, $(\Psi_{\phi, \mathbf{Q},
    \mathbf{P}})^{-1}=\Psi_{\phi^{-1},\mathbf{Q}^{-1},-\mathbf{M^{-1}PQ^{-1}}}$, where
    $\mathbf{M}\in \GL_n(\ZZ)$ is the abelianization of $\phi$.
\end{itemize}
\end{prop}

\begin{proof}
(i). Suppose that $\Psi$ is injective. Then $\Psi$ can not be of \tii\ since, if it
were, the commutator of any two elements in $F_n$ ($n\geqslant 2$) would be in the kernel
of $\Psi$. Hence, $\Psi =\Psi_{\phi, \mathbf{Q}, \mathbf{P}}$ for some $\phi \in
\End(F_n)$, $\mathbf{Q}\in \mathcal{M}_{m}(\ZZ)$, and $\mathbf{P}\in \mathcal{M}_{n\times
m}(\ZZ)$. Since $\mathbf{t^a}\Psi =\mathbf{t^{a Q}}$, the injectivity of $\Psi$ implies
that of $\mathbf{a}\mapsto \mathbf{aQ}$; hence, $\det(\mathbf{Q})\neq 0$. Finally, in order
to prove the injectivity of $\phi$, let $u\in F_n$ with $u\phi =1$. Note that the
endomorphism of $\QQ^m$ given by $\mathbf{Q}$ is invertible so, in particular, there exist
$\mathbf{v}\in \QQ^m$ such that $\mathbf{v}\mathbf{Q} =\mathbf{uP}$; write
$\mathbf{v}=\frac{1}{b} \mathbf{a}$ for some $\mathbf{a}\in \ZZ^m$ and $b\in \ZZ$, $b\neq
0$, and we have $\mathbf{aQ}=b\, \mathbf{vQ}=b\, \mathbf{uP}$; thus,
$(\mathbf{t}^{\mathbf{a}}u^{-b})\Psi
=\mathbf{t}^{\mathbf{aQ}}(\mathbf{t}^{\mathbf{uP}}1)^{-b}
=\mathbf{t}^{\mathbf{aQ-}b\mathbf{uP}}=\mathbf{t^0}=1$. Hence,
$\mathbf{t}^{\mathbf{a}}u^{-b}=1$ and so, $u=1$.

Conversely, let $\Psi =\Psi_{\phi, \mathbf{Q}, \mathbf{P}}$ be of \ti, with $\phi$ a
monomorphism of $F_n$ and $\det (\mathbf{Q})\neq \mathbf{0}$, and let $\mathbf{t^{a}}u\in
G$ be such that $1=(\mathbf{t^{a}}u)\Psi =\mathbf{t}^{\mathbf{aQ+uP}} \, u\phi$. Then,
$u\phi=1$ and so, $u=1$; and $\mathbf{0=aQ+uP=aQ}$ and so, $\mathbf{a=0}$. Hence, $\Psi$ is
injective.

(ii). Suppose that $\Psi$ is onto. Since the image of an endomorphism of \tii\ followed
by the projection $\pi$ onto $F_n$, $n\geqslant 2$, is contained in $\langle z\rangle$ (and
so is cyclic), $\Psi$ cannot be of \tii. Hence, $\Psi =\Psi_{\phi, \mathbf{Q},
\mathbf{P}}$ for some $\phi \in \End(F_n)$, $\mathbf{Q}\in \mathcal{M}_{m}(\ZZ)$, and
$\mathbf{P}\in \mathcal{M}_{n\times m}(\ZZ)$. Given $v\in F_n \leqslant G$ there must be
$\mathbf{t^a}u\in G$ such that $(\mathbf{t^a}u)\Psi =v$ and so $u\phi =v$. Thus $\phi
\colon F_n\to F_n$ is onto. On the other hand, for every $j\in [m]$, let
$\boldsymbol{\delta}_{\mathbf{j}}$ be the canonical vector of $\ZZ^m$ with 1 at coordinate
$j$, and let $\mathbf{t}^{\mathbf{b_j}} u_j\in G$ be a pre-image by $\Psi$ of $t_j
=\mathbf{t}^{\boldsymbol{\delta}_{\mathbf{j}}}$. We have
$(\mathbf{t}^{\mathbf{b_j}}u_j)\Psi =\mathbf{t}^{\boldsymbol{\delta}_{\mathbf{j}}}$, i.e.\
$u_j \phi =1$, $\mathbf{u_j}=\mathbf{0}$ and
$\mathbf{b_j}\mathbf{Q}=\mathbf{b_j}\mathbf{Q}+\mathbf{u_j}\mathbf{P}=\boldsymbol{\delta}_{\mathbf{j}}$.
This means that the matrix $\mathbf{B}$ with rows $\mathbf{b_j}$ satisfies
$\mathbf{BQ=I_m}$ and thus, $\det (\mathbf{Q})=\pm 1$.

Conversely, let $\Psi =\Psi_{\phi, \mathbf{Q}, \mathbf{P}}$ be of \ti, with $\phi$
being an epimorphism of $F_n$ and $\det (\mathbf{Q})=\pm 1$. By the hopfianity of $F_n$,
$\phi \in \Aut(F_n)$ and we can consider $\Upsilon
=\Psi_{\phi^{-1},\mathbf{Q}^{-1},-\mathbf{M^{-1}PQ^{-1}}}$, where $\mathbf{M}\in GL_n(\ZZ)$
is the abelianization of $\phi$. For every $\mathbf{t^{a}}u\in G$, we have
 $$
(\mathbf{t^{a}}u)\Upsilon\Psi =\bigl(\mathbf{t^{aQ^{-1}-uM^{-1}PQ^{-1}}} (u\phi^{-1})\bigr)\Psi
=\mathbf{t^{a-uM^{-1}P+uM^{-1}P}}u =\mathbf{t^{a}}u.
 $$
Hence, $\Psi$ is onto.

(iii). The equivalence is a direct consequence of~(i) and~(ii). To see the actual value of
$\Psi^{-1}$ it remains to compute the composition in the reverse order:
\begin{equation*}
(\mathbf{t^{a}}u)\Psi\Upsilon =\bigl(\mathbf{t^{aQ+uP}}(u\phi)\bigr)\Upsilon =\mathbf{t^{a+uPQ^{-1}-uMM^{-1}PQ^{-1}}}u
=\mathbf{t^{a}}u. \qedhere
\end{equation*}
\end{proof}

Immediately from these characterizations for an endomorphism to be mono, epi or auto, we
have the following corollary.

\begin{cor}
$\ZZ^m \times F_n $ is hopfian and not cohopfian. \qed
\end{cor}

The hopfianity of free-abelian times free groups was already known as part of a bigger
result: in \cite{green_graph_1990} and~\cite{humphries_stephen_p._representations_1994} it
was shown that finitely generated partially commutative groups (this includes groups of the
form $G=\ZZ^m \times F_n$) are residually finite and so, hophian. However, our proof is
more direct and explicit in the sense of giving complete characterizations of the
injectivity and surjectivity of a given endomorphism of $G$. We remark that, despite it
could seem reasonable, the hophianity of $\ZZ^m \times F_n$ does not follow directly from
that of free-abelian and free groups (both very well known): in~\cite{tyrer_direct_1971},
the author constructs a direct product of two hophian groups which is \emph{not} hophian.

For later use, next lemma summarizes how to operate \ti\ endomorphisms (compose, invert
and take a power); it can be easily proved by following routine computations. The reader
can easily find similar equations for the composition of two \tii\ endomorphisms, or
one of each  (we do not include them here because they will not be necessary for the rest
of the paper).

\begin{lem} \label{lab:endosI comportament algebraic}
Let $\Psi_{\phi,\mathbf{Q},\mathbf{P}}$ and $\Psi_{\phi',\mathbf{Q'},\mathbf{P'}}$ be two
\ti\ endomorphisms of $G=\ZZ^m \times F_n$, $n\neq 1$, and denote by $\mathbf{M} \in
\mathcal{M}_n(\ZZ)$ the (matrix of the) abelianization of $\phi \in \End(F_n)$. Then,
\begin{itemize}
\item[\emph{(i)}] $\Psi_{\phi,\mathbf{Q},\mathbf{P}} \cdot
    \Psi_{\phi',\mathbf{Q'},\mathbf{P'}} =\Psi_{\phi \phi',\mathbf{QQ'},
    \mathbf{PQ'}+\mathbf{M}\mathbf{P'}}$,
\item[\emph{(ii)}] for all $k\geqslant 1$, $(\Psi_{\phi,\mathbf{Q},\mathbf{P}})^k=
    \Psi_{\phi^k, \mathbf{Q}^k, \mathbf{P_k}}$, where $\mathbf{P_k}=\sum_{i=1}^{k}
    \mathbf{M}^{i-1} \mathbf{P} \mathbf{Q}^{k-i}$,
\item[\emph{(iii)}] $\Psi_{\phi,\mathbf{Q},\mathbf{P}}$ is invertible if and only if
    $\phi \in \Aut(F_n)$ and $\mathbf{Q}\in \GL_m(\ZZ)$; in this case,
    $(\Psi_{\phi,\mathbf{Q},\mathbf{P}})^{-1}= \Psi_{\phi^{-1}, \mathbf{Q}^{-1},
    -\mathbf{M^{-1}PQ^{-1}}}$.
\item[\emph{(iv)}] For every $\mathbf{a}\in \ZZ^m$ and $u\in F_n$, the right conjugation
    by $\mathbf{t^{a}}u$ is $\Gamma_{\mathbf{t^{a}}u}
    =\Psi_{\gamma_{u},\mathbf{I_m},\mathbf{0}}$, where $\gamma_u$ is the right
    conjugation by $u$ in $F_n$, $v\mapsto u^{-1}vu$, $\mathbf{I_m}$ is the identity
    matrix of size~$m$, and $\mathbf{0}$ is the zero matrix of size $n\times m$. \qed
\end{itemize}
\end{lem}

In the rest of the section, we shall use this information to derive the structure of
$\Aut(G)$, where $G=\ZZ^m \times F_n$, $m\geqslant 1$, $n\geqslant 2$.

\begin{thm}\label{prop:desc autos}
For $G=\ZZ^m \times F_n$, with $m\geqslant 1$ and $n\geqslant 2$, the group $\Aut (G)$ is
isomorphic to the semidirect product $\mathcal{M}_{n\times m}(\ZZ) \rtimes (\Aut(F_n)
\times \GL_m(\ZZ))$ with respect to the natural action. In particular, $\Aut(G)$ is
finitely presented.
\end{thm}

\begin{proof}
First or all note that, for every $\phi,\, \phi' \in \Aut(F_n)$, every $\mathbf{Q},\,
\mathbf{Q'}\in \GL_m(\ZZ)$, and every $\mathbf{P},\mathbf{P}'\in \mathcal{M}_{n\times
m}(\ZZ)$, we have
 $$
\Psi_{\phi, \mathbf{I_m}, \mathbf{0}} \cdot \Psi_{\phi', \mathbf{I_m}, \mathbf{0}}=
\Psi_{\phi\phi', \mathbf{I_m}, \mathbf{0}},
 $$
 $$
\Psi_{I_n, \mathbf{Q}, \mathbf{0}} \cdot \Psi_{I_n, \mathbf{Q'}, \mathbf{0}} =\Psi_{I_n,
\mathbf{QQ'}, \mathbf{0}}
 $$
 $$
\Psi_{I_n, \mathbf{I_m},\mathbf{P}} \cdot \Psi_{I_n, \mathbf{I_m},\mathbf{P'}}=\Psi_{I_n,
\mathbf{I_m},\mathbf{P+P'}}.
 $$
Hence, the three groups $\Aut(F_n)$, $\GL_m(\ZZ)$, and $\mathcal{M}_{n\times m}(\ZZ)$ (this
last one with the addition of matrices), are all subgroups of $\Aut(G)$ via the three
natural inclusions: $\phi \mapsto \Psi_{\phi, \mathbf{I_m}, \mathbf{0}}$,\phantom{a}
$\mathbf{Q}\mapsto \Psi_{I_n, \mathbf{Q}, \mathbf{0}}$, and $\mathbf{P}\mapsto \Psi_{I_n,
\mathbf{I_m},\mathbf{P}}$, respectively. Furthermore, for every $\phi \in \Aut(F_n)$ and
every $\mathbf{Q}\in \GL_m(\ZZ)$, it is clear that $\Psi_{\phi, \mathbf{I_m}, \mathbf{0}}
\cdot \Psi_{I_n, \mathbf{Q}, \mathbf{0}} =\Psi_{I_n, \mathbf{Q}, \mathbf{0}} \cdot
\Psi_{\phi, \mathbf{I_m}, \mathbf{0}}$; hence $\Aut(F_n) \times \GL_m(\ZZ)$ is a subgroup
of $\Aut(G)$ in the natural way.

On the other hand, for every $\phi \in \Aut(F_n)$, every $\mathbf{Q}\in \GL_m(\ZZ)$, and
every $\mathbf{P}\in \mathcal{M}_{n\times m}(\ZZ)$, we have
\begin{equation}\label{eq:conjFn}
(\Psi_{\phi, \mathbf{I_m}, \mathbf{0}})^{-1} \cdot \Psi_{I_n, \mathbf{I_m},\mathbf{P}} \cdot \Psi_{\phi, \mathbf{I_m}, \mathbf{0}} = \Psi_{\phi^{-1}, \mathbf{I_m}, \mathbf{0}} \cdot \Psi_{\phi, \mathbf{I_m},\mathbf{P}} = \Psi_{I_n, \mathbf{I_m},\mathbf{M^{-1}P}},
\end{equation}
where $\mathbf{M}\in \GL_n(\ZZ)$ is the abelianization of $\phi$, and
\begin{equation}\label{eq:conjQ}
(\Psi_{I_n, \mathbf{Q}, \mathbf{0}})^{-1} \cdot \Psi_{I_n, \mathbf{I_m},\mathbf{P}} \cdot \Psi_{I_n, \mathbf{Q}, \mathbf{0}} = \Psi_{I_n, \mathbf{Q^{-1}}, \mathbf{0}} \cdot \Psi_{I_n, \mathbf{Q},\mathbf{PQ}} = \Psi_{I_n, \mathbf{I_m},\mathbf{PQ}}.
\end{equation}
In particular, $\mathcal{M}_{n\times m}(\ZZ)$ is a normal subgroup of $\Aut (G)$. But
$\Aut(F_n)$, $\GL_m(\ZZ)$ and $\mathcal{M}_{n\times m}(\ZZ)$ altogether generated the whole
$\Aut (G)$, as can be seen with the equality
\begin{equation} \label{eq:desc autos}
\Psi_{\phi,\mathbf{Q},\mathbf{P}}= \Psi_{I_n, \mathbf{I_m}, \mathbf{PQ^{-1}}}\cdot \Psi_{I_n, \mathbf{Q},
\mathbf{0}} \cdot \Psi_{\phi,\mathbf{I_m},\mathbf{0}}.
\end{equation}
Thus, $\Aut(G)$ is isomorphic to the semidirect product $\mathcal{M}_{n\times m}(\ZZ)
\rtimes (\Aut(F_n) \times \GL_m(\ZZ))$, with the action of $\Aut(F_n) \times \GL_m(\ZZ)$ on
$\mathcal{M}_{n\times m}(\ZZ)$ given by equations~\eqref{eq:conjFn} and~\eqref{eq:conjQ}.

But it is well known that these three groups are finitely presented: $\mathcal{M}_{n\times
m}(\ZZ) \simeq \ZZ^{nm}$ is free-abelian generated by canonical matrices (with zeroes
everywhere except for one position where there is a 1), $\GL_m(\ZZ)$ is generated by
elementary matrices, and $\Aut(F_n)$ is generated, for example, by the Nielsen
automorphisms (see~\cite{lyndon_combinatorial_2001} for details and full finite
presentations). Therefore, $\Aut(G)$ is also finitely presented (and one can easily obtain
a presentation of $\Aut(G)$ by taking together the generators for $\mathcal{M}_{n\times
m}(\ZZ)$, $\Aut(F_n)$ and $\GL_m(\ZZ)$, and putting as relations those of each of
$\mathcal{M}_{n\times m}(\ZZ)$, $\Aut(F_n)$ and $\GL_m(\ZZ)$, together with the commutators
of all generators from $\Aut(F_n)$ with all generators from $\GL_m(\ZZ)$, and with the
conjugacy relations describing the action of $\Aut(F_n)\times \GL_m(\ZZ)$ on
$\mathcal{M}_{n\times m}(\ZZ)$ analyzed above).
\end{proof}

Finite presentability of $\Aut(G)$ was previously known as a particular case of a more
general result: in~\cite{laurence_generating_1995}, M. Laurence gave a finite family of
generators for the group of automorphisms of any finitely generated partially commutative
group, in terms of the underlying graph. It turns out that, when particularizing this to
free-abelian times free groups, Laurence's generating set for $\Aut(G)$ is essentially the
same as the one obtained here, after deleting some obvious redundancy. Later,
in~\cite{day_peak_2009}, M. Day builts a kind of peak reduction for such groups, from which
he deduces finite presentation for its group of automorphisms. However, our
Theorem~\ref{prop:desc autos} is better in the sense that it provides the explicit
structure of the automorphism group of a free-abelian times free group.

\goodbreak

\section{The subgroup fixed by an endomorphism}\label{sec:fix}

In this section we shall study when the subgroup fixed by an endomorphism of $\ZZ^m \times
F_n$ is finitely generated and, in this case, we shall consider the problem of
algorithmically computing a basis for it. We will consider the following two problems.

\begin{problem}[\textbf{Fixed Point Problem, $\FPP_\mathbf{a}(G)$}]
Given an automorphism $\Psi$ of $G$ (by the images of the generators), decide whether $\Fix
\Psi$ is finitely generated and, if so, compute a set of generators for it.
\end{problem}

\begin{problem}[\textbf{Fixed Point Problem, $\FPP_\mathbf{e}(G)$}]
Given an endomorphism $\Psi$ of $G$ (by the images of the generators), decide whether $\Fix
\Psi$ is finitely generated and, if so, compute a set of generators for it.
\end{problem}

Of course, the fixed point subgroup of an arbitrary endomorphism of $\ZZ^m$ is finitely
generated, and the problems $\FPP_\mathbf{e}(\ZZ^m)$ and $\FPP_\mathbf{a}(\ZZ^m)$ are
clearly solvable, just reducing to solve the corresponding systems of linear equations.

Again, the case of free groups is much more complicated. Gersten showed
in~\cite{gersten_fixed_1987} that $\rank(\Fix \phi)<\infty$ for every automorphism $\phi
\in \Aut(F_n)$, and Goldstein and Turner~\cite{goldstein_fixed_1986} extended this result
to arbitrary endomorphisms of $F_n$.

About computability, O. Maslakova published~\cite{maslakova_fixed_2003} in 2003, giving an
algorithm to compute a free basis for $\Fix \phi$, where $\phi \in \Aut (F_n)$. After its
publication, the arguments were found to be incorrect. An attempt to fix
them and provide a correct solution to $\FPP_\mathbf{a}(F_n)$ has been recently made by O.
Bogopolski and O. Maslakova in the preprint~\cite{bogopolski_basis_2012} not yet published (see the
beginning of page 3); here, the arguments are quite involved and difficult, making strong
and deep use of the theory of train tracks. It is worth mentioning at this point that, this
problem was previously solved in some special cases with much simpler arguments and
algorithms (see, for example Cohen and Lustig~\cite{marshall_m._cohen_dynamics_1989} for positive automorphisms,
Turner~\cite{turner_finding_1995} for special irreducible automorphisms, and Bogopolski~\cite{bogopolski_classification_2000} for the
case $n=2$). On the other hand, the problem $\FPP_\mathbf{e}(F_n)$ remains still open in
general.

When one moves to free-abelian times free groups, the situation is even more involved.
Similar to what happens with respect to the Howson property, $\Fix \Psi$ need not be
finitely generated for $\Psi \in \Aut (\ZZ \times F_2)$, and essentially the same example
from Observation~\ref{prop:Fn x Z^n no Howson} can be recycled here: consider the \ti\
automorphism $\Psi$ given by $a\mapsto ta$, $b\mapsto b$, $t\mapsto t$; clearly, $t^rw(a,b)
\mapsto t^{r+|w|_a}w(a,b)$ and so,
 $$
\Fix \Psi =\{ t^r w(a,b) \,\, \vert \,\, |w|_a=0\} =\llangle t,\, b\rrangle =\langle t,
a^{-k}ba^k \,\ (k\in \ZZ) \rangle
 $$
is not finitely generated.

In the present section we shall analyze how is the fixed point subgroup of an endomorphism
of a free-abelian times free group, and we shall give an explicit characterization on when
is it finitely generated. In the case it is, we shall also consider the computability of a
finite basis for the fixed subgroup, and will solve the problems $\FPP_\mathbf{a}(\ZZ^m
\times F_n)$ and $\FPP_\mathbf{e}(\ZZ^m \times F_n)$ modulo the corresponding problems for
free groups, $\FPP_\mathbf{a}(F_n)$ and $\FPP_\mathbf{e}(F_n)$. (Our arguments descend
directly from $\End(\ZZ^m \times F_n)$ to $\End(F_n)$, in such a way that any partial
solution to the free problems can be used to give the corresponding partial solution to the
free-abelian times free problems, see Proposition~\ref{algo-I} below.)

Let us distinguish the two types of endomorphisms according to
Proposition~\ref{prop:classificacio endos} (and starting with the easier \tii\ ones).

\begin{prop}\label{prop:fix fg typeII}
Let $G=\ZZ^m \times F_n$ with $n\neq 1$, and consider a \tii\ endomorphism $\Psi$,
namely
 $$
\Psi = \Psi_{z,\mathbf{l,h,Q,P}} \colon \mathbf{t^a}u\mapsto \mathbf{t^{aQ+uP}}z^{\mathbf{a}\mathbf{l}\tr
+\mathbf{u}\mathbf{h}\tr},
 $$
where $1\neq z\in F_n$ is not a proper power, $\mathbf{Q}\in \mathcal{M}_{m}(\ZZ)$,
$\mathbf{P}\in \mathcal{M}_{n\times m}(\ZZ)$, $\mathbf{0}\neq \mathbf{l}\in \ZZ^m $, and
$\mathbf{h}\in \ZZ^n$. Then, $\Fix \Psi$ is finitely generated, and a basis for $\Fix \Psi$
is algorithmically computable.
\end{prop}

\begin{proof}
First note that $\im \Psi$ is an abelian subgroup of $\ZZ^m \times F_n$. Then, by Corollary
\ref{cor:classes isomorfia subgrups}, it must be isomorphic to $\ZZ^{m'}$ for a certain
$m'\leqslant m+1$. Therefore, $\Fix \Psi \leqslant \im (\Psi)$ is isomorphic to a subgroup
of $\ZZ^{m'}$, and thus finitely generated.

According to the definition, an element $\mathbf{t}^\mathbf{a}u$ is fixed by $\Psi$ if and
only if $\mathbf{t^{aQ+uP}}z^{\mathbf{a}\mathbf{l}\tr +\mathbf{u}\mathbf{h}\tr}
=\mathbf{t}^\mathbf{a}u$. For this to be satisfied, $u$ must be a power of $z$, say $u=z^r$
for certain $r\in \ZZ$, and abelianizing we get $\mathbf{u}=r \mathbf{z}$, and the system
of equations
\begin{equation} \label{eq:system fix typeII ab}
\left.
\begin{aligned}
\mathbf{a}\mathbf{l}\tr + r \mathbf{z}\mathbf{h}\tr &= r \\
\mathbf{a}(\mathbf{I_m}-\mathbf{Q})& = r \mathbf{z} \mathbf{P}
\end{aligned}
\ \right\}
\end{equation}
whose set $\mathcal{S}$ of integer solutions $(\mathbf{a},r) \in \ZZ^{m+1}$ describe
precisely the subgroup of fixed points by~$\Psi$:
 $$
\Fix \Psi =\{ \mathbf{t^a} z^r \mid (\mathbf{a},r)\in \mathcal{S} \}.
 $$
By solving~\eqref{eq:system fix typeII ab}, we get the desired basis for $\Fix \Psi$. The
proof is complete.
\end{proof}

\goodbreak

\begin{thm}\label{prop:fix fg typeI}
Let $G=\ZZ^m \times F_n$ with $n\neq 1$, and consider a \ti\ endomorphism $\Psi$,
namely
 $$
\Psi=\Psi_{\phi,\mathbf{Q,P}} \colon \mathbf{t^a}u\mapsto \mathbf{t^{aQ+uP}}\, u\phi,
 $$
where $\phi \in \End(F_n)$, $\mathbf{Q}\in \mathcal{M}_{m}(\ZZ)$, and $\mathbf{P}\in
\mathcal{M}_{n\times m}(\ZZ)$. Let $N=\im (\mathbf{I_m-Q})\cap \im \mathbf{P'}$, where
$\mathbf{P'}$ is the restriction of $\mathbf{P}\colon \ZZ^n \to \ZZ^m$ to $(\Fix \phi
)\rho$, the image of $\Fix \phi \leqslant F_n$ under the global abelianization $\rho \colon
F_n \twoheadrightarrow \ZZ^n$. Then, $\Fix \Psi$ is finitely generated if and only if one
of the following happens: \emph{(i)}~$\Fix \phi =1$; \emph{(ii)} $\Fix \phi$ is cyclic,
$(\Fix \phi )\rho \neq \{ \mathbf{0}\}$, and $N\mathbf{P'}^{-1} =\{ \mathbf{0}\}$; or
\emph{(iii)} $\rank (N)=\rank (\im \mathbf{P'})$.
\end{thm}

\begin{proof}
An element $\mathbf{t}^\mathbf{a}u$ is fixed by $\Psi$ if and only if
$\mathbf{t}^{\mathbf{aQ+uP}}u\phi =\mathbf{t}^\mathbf{a}u$, i.e.\ if and only if
 \begin{equation*}
\left.
\begin{aligned}
u\phi &= u \\
\mathbf{a}(\mathbf{I_m}-\mathbf{Q})&=\mathbf{uP}
\end{aligned}
\ \right\}
\end{equation*}
That is,
\begin{equation}\label{fixPsi}
\Fix \Psi =\{ \mathbf{t}^\mathbf{a}u \in G \mid u \in \Fix \phi \text{\, and \,} \mathbf{a}(\mathbf{I_m}-\mathbf{Q})=
\mathbf{uP}\},
\end{equation}
where $\mathbf{u}=u\rho$, and $\rho \colon F_n \twoheadrightarrow \ZZ^n$ is the
abelianization map. As we have seen in Corollary~\ref{cor:H fg sii Hpi fg}, $\Fix \Psi$ is
finitely generated if and only if its projection to the free part
\begin{equation}\label{fix}
(\Fix \Psi)\pi =\Fix \phi \, \cap \, \{ u\in F_n \mid \mathbf{uP}\in \im(\mathbf{I_m}-\mathbf{Q})\}
\end{equation}
is so. Now (identifying integral matrices $\mathbf{A}$ with the corresponding linear
mapping $\mathbf{v} \mapsto \mathbf{vA}$, as usual), let $M$ be the image of
$\mathbf{I_m}-\mathbf{Q}$, and consider its preimage first by $\mathbf{P}$ and then by
$\rho$, see the following diagram:
\begin{equation*} \index{punts fixos per un automorfisme!diagrames}
\xy
(55.5,-9.1)*+{= \im (\mathbf{I_m}-\mathbf{Q}) .};
(0,-4.5)*+{\rotatebox[origin=c]{90}{$\normaleq$}};
(22,-4.5)*+{\rotatebox[origin=c]{90}{$\normaleq$}};
(42,-4.5)*+{\rotatebox[origin=c]{90}{$\normaleq$}};
(-4,0)*+{\leqslant};
(-10.5,0)*+{\Fix \phi};
{\ar@{->>}^-{\rho} (0,0)*++{F_{n}}; (22,0)*++{\ZZ^{n}}};
{\ar^-{\mathbf{P}} (22,0)*++++{}; (42,0)*++{\ZZ^m}};
{\ar@(ur,ul)_{\mathbf{I_m-Q}} (42,0)*++{}; (42,0)*++{} };
{\ar@{|->} (42,-9)*++{M }; (22,-9)*++{M \mathbf{P}^{-1}}};
{\ar@{|->} (22,-9)*++++++{}; (0,-9)*+{M \mathbf{P}^{-1} \rho^{-1} }};
\endxy
\end{equation*}
Equation~(\ref{fix}) can be rewritten as
\begin{equation}\label{fix2}
(\Fix \Psi) \pi =\Fix \phi \, \cap \, M\mathbf{P}^{-1} \rho^{-1}.
\end{equation}
However, this description does not show whether $\Fix \Psi$ is finitely generated because
$\Fix \phi$ is in fact finitely generated, but $M\mathbf{P}^{-1}\rho^{-1}$ is not in
general. We shall avoid the intersection with $\Fix \phi$ by reducing $M$ to a certain
subgroup. Let $\rho'$ be the restriction of $\rho$ to $\Fix \phi$ (not to be confused with
the abelianization map of the subgroup $\Fix \phi$ itself), let $\mathbf{P'}$ be the
restriction of $\mathbf{P}$ to $\im \rho'$, and let $N=M\cap \im \mathbf{P'}$, see the
following diagram:
\begin{equation*}\index{punts fixos per un automorfisme!diagrames}
 \xy
 (65,0)*+{\geqslant M = \im (\mathbf{I_m}-\mathbf{Q})};
 (58,-18.1)*+{= M \cap \im \mathbf{P'} .};
 (0,-4.5)*+{\rotatebox[origin=c]{90}{$\leqslant$}};
 (24,-4.5)*+{\rotatebox[origin=c]{90}{$\normaleq$}};
 (46,-4.5)*+{\rotatebox[origin=c]{90}{$\normaleq$}};
 {\ar@{->>}^-{\rho} (0,0)*++{F_{n}}; (24,0)*++{\ZZ^{n}}};
 {\ar^-{\mathbf{P}} (24,0)*++++{}; (46,0)*++{\ZZ^m}};
 {\ar@{->>}^-{\rho'} (0,-9)*++{\Fix \phi}; (24,-9)*++{\im \rho'}};
 {\ar@{->>}^-{\mathbf{P'}} (24,-9)*+++++{}; (46,-9)*++{\im \mathbf{P'}}};
 (0,-13.5)*+{\rotatebox[origin=c]{90}{$\normaleq$}};
 (24,-13.5)*+{\rotatebox[origin=c]{90}{$\normaleq$}};
 (46,-13.5)*+{\rotatebox[origin=c]{90}{$\normaleq$}};
 {\ar@(ur,ul)_{\mathbf{I_m-Q}} (46,0)*++{}; (46,0)*++{} };
 {\ar@{|->} (46,-18)*+++{N }; (24,-18)*++{N \mathbf{P'}^{-1}}};
 {\ar@{|->} (24,-18)*+++++++{}; (0,-18)*+{N \mathbf{P'}^{-1} \rho'^{-1} }};
 (0,-23)*+{\rotatebox[origin=c]{90}{$=$}};
 (0,-28)*+{(\Fix \Psi) \pi};
 \endxy
\end{equation*}
Equation~(\ref{fix2}) then rewrites into
 $$
(\Fix \Psi )\pi =N\mathbf{P'}^{-1}\rho'^{-1}.
 $$
Now, since $N\mathbf{P'}^{-1}\rho'^{-1}$ is a normal subgroup of $\Fix \phi$ (not, in
general, of $F_n$), it is finitely generated if and only if it is either trivial, or of
finite index in $\Fix \phi$.

Note that $\rho'$ is injective (and thus bijective) if and only if $\Fix \phi$ is either
trivial, or cyclic not abelianizing to zero (indeed, for this to be the case we cannot have
two freely independent elements in $\Fix \phi$ and so, $\rank (\Fix \phi)\leqslant 1$).
Thus, $(\Fix \Psi )\pi =N\mathbf{P'}^{-1}\rho'^{-1}=1$ if and only if $\Fix \phi$ is
trivial or cyclic not abelianizing to zero, and $N\mathbf{P'}^{-1}=\{\mathbf{0}\}$.

On the other side, by Lemma~\ref{lem:index de subgrups a traves de epimorfismes}~(ii),
$N\mathbf{P'}^{-1}\rho'^{-1}$ has finite index in $\Fix \phi$ if and only if $N$ has finite
index in $\im \mathbf{P'}$ i.e.\ if and only if $\rank (N)=\rank (\im \mathbf{P'})$.
\end{proof}

\begin{exm}
Let us analyze again the example given at the beginning of this section, under the light of
the Theorem~\ref{prop:fix fg typeI}. We considered the automorphism $\Psi$ of $\ZZ\times
F_2=\langle t \mid \, \rangle \times \langle a,b \mid \, \rangle$ given by $a\mapsto ta$,
$b\mapsto b$ and $t\mapsto t$, i.e.\ $\Psi =\Psi_{I_2, \mathbf{I_1}, \mathbf{P}}$, where
$\mathbf{P}$ is the $2\times 1$ matrix $\mathbf{P=(1,0)}\tr$. It is clear that $\Fix
(I_2)=F_2$ and so, conditions~(i) and~(ii) from Proposition~\ref{prop:fix fg typeI} do not
hold. Furthermore, $\rho'=\rho$, $\mathbf{P'=P}$, $M=\im (\mathbf{0})=\{ \mathbf{0}\}$,
$N=\{ \mathbf{0}\}$, while $\im \mathbf{P'}=\mathbb{Z}$; hence, condition~(iii) from
Theorem~\ref{prop:fix fg typeI} does not hold either, according to the fact that $\Fix
\Psi$ is not finitely generated.
\end{exm}

\goodbreak

Finally, the proof of Theorem~\ref{prop:fix fg typeI} is explicit enough to allow us to
make the whole thing algorithmic:  given a \ti\ endomorphism $\Psi
=\Psi_{\phi,\mathbf{Q,P}} \in \End (\ZZ^m \times F_n)$, the decision on whether $\Fix \Psi$
if finitely generated or not, and the computation of a basis for it in case it is, can be
made effective assuming we have a procedure to compute a (free) basis for $\Fix \phi$:

\begin{prop}\label{algo-I}
Let $G=\ZZ^m \times F_n$ with $n\neq 1$, and let $\Psi=\Psi_{\phi,\mathbf{Q,P}}$ be a
\ti\ endomorphism of $G$. Assuming a (finite and free) basis for $\Fix \phi$ is given
to us, we can algorithmically decide whether $\Fix \Psi$ is finitely generated or not and,
in case it is, compute a basis for it.
\end{prop}

\begin{proof}
Let $\{ v_1,\ldots ,v_p \}$ be the (finite and free) basis for $\Fix \phi \leqslant F_n$
given to us in the hypothesis.

Theorem~\ref{prop:fix fg typeI} describes how is $\Fix \Psi$ and when is it finitely
generated. Assuming the notation from the proof there, we can compute abelian bases for
$N\leqslant \im \mathbf{P'}\leqslant \ZZ^m$ and $N\mathbf{P'}^{-1}\leqslant \im
\rho'\leqslant \ZZ^n$. Then, we can easily check whether any of the following three
conditions hold:
\begin{itemize}
\item[(i)] $\Fix \phi$ is trivial,
\item[(ii)] $\Fix \phi =\langle z\rangle$, $z\rho \neq \mathbf{0}$ and
    $N\mathbf{P'}^{-1}=\{ \mathbf{0}\}$,
\item[(iii)] $\rank (N)=\rank (\im \mathbf{P'})$.
\end{itemize}
If (i), (ii) and (iii) fail then $\Fix \Psi$ is not finitely generated and we are done.
Otherwise, $\Fix \Psi$ is finitely generated and it remains to compute a basis.
From~\eqref{eq:factoritzacio subgrup}, we have
 $$
\Fix \Psi =\bigl((\Fix \Psi )\cap \ZZ^m \bigr)\, \times \, (\Fix \Psi)\pi \alpha,
 $$
where $\Fix \Psi \overset{\alpha}{\longleftarrow} (\Fix \Psi )\pi$ is any splitting of
$\pi_{\mid \Fix \Psi} \colon \Fix \Psi\twoheadrightarrow (\Fix \Psi )\pi$. We just have to
compute a basis for each part and put them together (after algorithmically computing some
splitting $\alpha$). Regarding the abelian part, equation~(\ref{fixPsi}) tells us that
 $$
(\Fix \Psi )\cap \ZZ^m =\{ \mathbf{t}^\mathbf{a} \mid \mathbf{a}(\mathbf{I_m}-\mathbf{Q})=\mathbf{0} \},
 $$
and we can easily find an abelian basis for it by just computing $\ker (\mathbf{I_m-Q})$.
\goodbreak
Consider now the free part. In cases (i) and (ii), $(\Fix \Psi )\pi =1$ and there is
nothing to compute. Note that, in these cases, $\Fix \Psi$ is then an abelian subgroup of
$\ZZ^m \times F_n$.

Assume case (iii), i.e.\ $\rank (N)=\rank (\im \mathbf{P'})$. In this situation, $N$ has
finite index in $\im \mathbf{P'}$ and so, $N\mathbf{P'}^{-1}$ has finite index in $\im
\rho'$; let us compute a set of coset representatives of $\im \rho'$
modulo~$N\mathbf{P'}^{-1}$,
 $$
\im \rho'=(N\mathbf{P'}^{-1})\mathbf{c_1}\sqcup \cdots \sqcup (N\mathbf{P'}^{-1})\mathbf{c_q},
 $$
(see Section~\ref{fi}). Now, according to Lemma~\ref{lem:index de subgrups a traves de
epimorfismes}~(b), we can transfer this partition via $\rho'$ to obtain a system of right
coset representatives of $\Fix \phi$ modulo $(\Fix \Psi )\pi =N\mathbf{P'}^{-1}\rho'^{-1}$,
\begin{equation}\label{cosets}
\Fix \phi =(N\mathbf{P'}^{-1}\rho'^{-1})z_1\sqcup \cdots \sqcup (N\mathbf{P'}^{-1}\rho'^{-1})z_q.
\end{equation}
To compute the $z_i$'s, note that $\mathbf{v_1}=v_1\rho',\, \ldots ,\,
\mathbf{v_p}=v_p\rho'$ generate $\im \rho'$, write each $\mathbf{c_i}\in \im \rho'$ as a
(non necessarily unique) linear combination of them, say
$\mathbf{c_i}=c_{i,1}\mathbf{v_1}+\cdots +c_{i,p}\mathbf{v_p}$, $i\in [q]$, and take $z_i
=v_1^{c_{i,1}} v_2^{c_{i,1}}\cdots v_{p}^{c_{i,p}}\in \Fix \phi$.

Now, construct a free basis for $N\mathbf{P'}^{-1}\rho'^{-1} =(\Fix \Psi )\pi$ following
the first of the two alternatives at the end of the proof of Theorem~\ref{thm:ip} (the
second one does not work here because $\rho'$ is not the abelianization of the subgroup
$\Fix \phi$, but the restriction there of the abelianization of $F_n$):

Build the Schreier graph $\mathcal{S}(N\mathbf{P'}^{-1}\rho'^{-1})$ for
$N\mathbf{P'}^{-1}\rho'^{-1}\leqslant \Fix \phi$ with respect to $\{v_1, \ldots ,v_{p}\}$,
in the following way: consider the graph with the cosets of~(\ref{cosets}) as vertices, and
with no edge. Then, for every vertex $(N\mathbf{P'}^{-1}\rho'^{-1})z_i$ and every letter
$v_j$, add an edge labeled $v_j$ from $(N\mathbf{P'}^{-1}\rho'^{-1})z_i$ to
$(N\mathbf{P'}^{-1}\rho'^{-1})z_i v_j$, algorithmically identified among the available
vertices by repeatedly using the membership problem for $N\mathbf{P'}^{-1}\rho'^{-1}$ (note
that we can easily do this by abelianizing the candidate and checking whether it belongs to
$N\mathbf{P'}^{-1}$). Once we have run over all $i,j$, we shall get the full graph
$\mathcal{S}(N\mathbf{P'}^{-1}\rho'^{-1})$, from which we can easily obtain a free basis
for $N\mathbf{P'}^{-1}\rho'^{-1} =(\Fix \Psi )\pi$.

Finally, having a free basis for $(\Fix \Psi )\pi$, we can easily construct an splitting
${\Fix \Psi \overset{\alpha}{\longleftarrow} (\Fix \Psi )\pi}$ for $\pi_{\mid \Fix \Psi}
\colon \Fix \Psi\twoheadrightarrow (\Fix \Psi )\pi$ by just computing, for each generator
$u\in (\Fix \Psi )\pi$, a preimage $\mathbf{t^a} u\in \Fix \Psi$, where $\mathbf{a}\in
\ZZ^m$ is a completion found by solving the system of equations $\mathbf{a(I_m-Q)=uP}$
(see~\eqref{fixPsi}).

This completes the proof.
\end{proof}

Bringing together Propositions~\ref{prop:fix fg typeII} and \ref{algo-I} and
Theorem~\ref{prop:fix fg typeI}, we get the following.

\begin{cor} \label{thm:punts fixos d'un endomorfisme}
For $m\geqslant 1$ and $n\geqslant 2$,
\begin{itemize}
\item[\emph{(i)}] if $\FPP_\mathbf{a}(F_n)$ is solvable then
    $\FPP_\mathbf{a}(\mathbb{Z}^m \times F_n)$ is also solvable.
\item[\emph{(ii)}] if $\FPP_\mathbf{e}(F_n)$ is solvable then
    $\FPP_\mathbf{e}(\mathbb{Z}^m \times F_n)$ is also solvable. $\Box$
\end{itemize}
\end{cor}

\goodbreak

To close this section, we point the reader to some very recent results related to fixed
subgroups of endomorphisms of partially commutative groups. In~\cite{rodaro_fixed_2012}, E. Rodaro, P.V.
Silva and M. Sykiotis characterize which partially commutative groups $G$ satisfy that
$\Fix \Psi$ is finitely generated for every $\Psi \in \End (G)$ (and, of course,
free-abelian times free groups are not included there); they also provide similar results
concerning automorphisms.

\section{The Whitehead problem}\index{problemes algor\'{\i}smics!problemes de Whitehead}
\label{sec:Wh}

J. Whitehead, back in the 30's of the last century, gave an
algorithm~\cite{whitehead_equivalent_1936} to decide, given two elements $u$ and $v$ from a
finitely generated free group $F_n$, whether there exists an automorphism $\phi \in \Aut
(F_n)$ sending one to the other, $v=u\phi$. Whitehead's algorithm uses a (today) very
classical piece of combinatorial group theory technique called `peak reduction', see
also~\cite{lyndon_combinatorial_2001}. Several variations of this problem (like replacing
$u$ and $v$ by tuples of words, relaxing equality to equality up to conjugacy, adding
conditions on the conjugators, replacing words by subgroups, replacing automorphisms to
monomorphisms or endomorphisms, etc), as well as extensions of all these problems to other
families of groups, can be found in the literature, all of them generally known as the
\emph{Whitehead problem}. Let us consider here the following ones for an arbitrary finitely
generated group~$G$:

\begin{problem}[\textbf{Whitehead Problem, $\WhP_\mathbf{a}(G)$}]
Given two elements $u,\, v\in G$, decide whether there exist an automorphism $\phi$ of $G$
such that $u\phi =v$, and, if so, find one (giving the images of the generators).
\end{problem}

\begin{problem}[\textbf{Whitehead Problem, $\WhP_\mathbf{m}(G)$}]
Given two elements $u,\, v\in G$, decide whether there exist a monomorphism $\phi$ of $G$
such that $u\phi =v$, and, if so, find one (giving the images of the generators).
\end{problem}

\begin{problem}[\textbf{Whitehead Problem, $\WhP_\mathbf{e}(G)$}]
Given two elements $u,\, v\in G$, decide whether there exist an endomorphism $\phi$ of $G$
such that $u\phi =v$, and, if so, find one (giving the images of the generators).
\end{problem}

In this last section we shall solve these three problems for free-abelian times free
groups. We note that, very recently, a new version of the classical peak-reduction theorem
has been developed by M. Day~\cite{day_full-featured_2012} for an arbitrary partially commutative group, see
also~\cite{day_peak_2009}. These techniques allow the author to solve the Whitehead problem
for partially commutative groups, in its variant relative to automorphisms and tuples of
conjugacy classes. In particular $\WhP_\mathbf{a}(G)$ (which was conjectured
in~\cite{day_peak_2009}) is solved in~\cite{day_full-featured_2012} for any partially commutative group $G$. As
far as we know, $\WhP_\mathbf{m}(G)$ and $\WhP_\mathbf{e}(G)$ remain unsolved in general.
Our Theorem~\ref{thm:Whitehead} below is a small contribution into this direction, solving
these problems for free-abelian times free groups.

Let us begin by reminding the situation of the Whitehead problems for free-abelian and for
free groups (the first one is folklore, and the second one is well-known). The following
lemma is straightforward to prove and, in particular, solves $\WhP_\mathbf{a}(\ZZ^m)$,
$\WhP_\mathbf{m}(\ZZ^m)$ and $\WhP_\mathbf{e}(\ZZ^m)$. Here, for a vector
$\mathbf{a}=(a_1,\ldots,a_m) \in \ZZ^m$, we write $\mcd(\mathbf{a})$ to denote the greatest
common divisor of the $a_i$'s (with the convention that $\mcd(\mathbf{0})=0$).

\begin{lem}\label{lem:caract uP i aQ}
If $\mathbf{u}\in \ZZ^n$ and $\mathbf{a}\in \ZZ^m \setminus \{\mathbf{0}\}$, then
\begin{enumerate}
\item[\emph{(i)}] $\{ \mathbf{aQ} \mid \mathbf{Q} \in \GL_m (\ZZ)\} =\{ \mathbf{a'}\in
    \ZZ^m \mid \mcd(\mathbf{a}) =\mcd(\mathbf{a'}) \}$,
\item[\emph{(ii)}] $\{ \mathbf{aQ} \mid \mathbf{Q} \in \mathcal{M}_m (\ZZ) \text{ with }
    \det(\mathbf{Q}) \neq 0 \} =\{ \mathbf{a'}\in \ZZ^m \mid \mcd(\mathbf{a}) \divides
    \mcd(\mathbf{a'}) \} \setminus \{\mathbf{0}\}$,
\item[\emph{(iii)}] $\{ \mathbf{uP} \mid \mathbf{P}\in \mathcal{M}_{n\times m} (\ZZ)\}
    =\{ \mathbf{u'}\in \ZZ^m \mid \mcd(\mathbf{u}) \mid \mcd(\mathbf{u'}) \}$. \qed
\end{enumerate}
\end{lem}

As expected, the same problems for the free group $F_n$ are much more complicated. As
mentioned above, the case of automorphisms was already solved by Whitehead back in the 30's
of last century. The case of endomorphisms can be solved by writing a system of equations
over $F_n$ (with unknowns being the images of a given free basis for $F_n$), and then
solving it by the powerful Makanin's algorithm. Finally, the case of monomorphisms was
recently solved by Ciobanu-Houcine.

\begin{thm}\label{thm:Whitehead free}
For $n\geqslant 2$,
\begin{itemize}
\item[\emph{(i)}] \emph{[\textbf{Whitehead, \cite{whitehead_equivalent_1936}}]}
    $\WhP_\mathbf{a}(F_n)$ is solvable.
\item[\emph{(ii)}] \emph{[\textbf{Ciobanu-Houcine, \cite{ciobanu_monomorphism_2010}}]}
    $\WhP_\mathbf{m}(F_n)$ is solvable.
\item[\emph{(iii)}] \emph{[\textbf{Makanin, \cite{makanin_equations_1982}}]}
    $\WhP_\mathbf{e}(F_n)$ is solvable. \qed
\end{itemize}
\end{thm}

\begin{thm}\label{thm:Whitehead}
Let $m\geqslant 1$ and $n\geqslant 2$, then
\begin{itemize}
\item[\emph{(i)}] $\WhP_\mathbf{a}(\ZZ^m \times F_n)$ is solvable.
\item[\emph{(ii)}] $\WhP_\mathbf{m}(\ZZ^m \times F_n)$ is solvable.
\item[\emph{(iii)}] $\WhP_\mathbf{e}(\ZZ^m \times F_n)$ is solvable.
\end{itemize}
\end{thm}

\begin{proof}
We are given two elements $\mathbf{t^a}u,\, \mathbf{t^b}v \in G=\ZZ^m \times F_n$, and have
to decide whether there exists an automorphism (resp. monomorphism, endomorphism) of $G$ sending one to the other. And in the affirmative case, find one of them. For convenience, we shall prove (ii), (i) and (iii) in this order.

(ii). Since all monomorphisms of $G$ are of \ti, we have to decide whether there exist a monomorphism $\phi$ of $F_n$, and matrices $\mathbf{Q} \in \mathcal{M}_m(\ZZ)$ and
$\mathbf{P}\in \mathcal{M}_{n\times m}(\ZZ)$, with $\det \mathbf{Q} \neq 0$, such that
$(\mathbf{t^a}u)\Psi_{\phi, \mathbf{Q}, \mathbf{P}} =\mathbf{t^b}v$. Separating the free
and free-abelian parts, we get two independent problems:
\begin{equation} \label{eq:Whitehead typeI}
\left.
\begin{aligned}
u\phi =v \\ \mathbf{aQ}+\mathbf{uP}=\mathbf{b}
\end{aligned}
\ \right\}
\end{equation}
On one hand, we can use Theorem~\ref{thm:Whitehead free}~(ii) to decide whether there
exists a monomorphism $\phi$ of $F_n$ such that $u\phi =v$. If not then our problem has no
solution either, and we are done; otherwise, $\WhP_\mathbf{m}(F_n)$ gives us such a $\phi$.

On the other hand, we need to know whether there exist matrices $\mathbf{Q} \in
\mathcal{M}_m(\ZZ)$ and ${\mathbf{P}\in \mathcal{M}_{n\times m}(\ZZ)}$, with $\det
\mathbf{Q} \neq 0$ and such that $\mathbf{aQ}+\mathbf{uP}=\mathbf{b}$, where $\mathbf{u}\in
\ZZ^n$ is the abelianization of $u\in F_n$ (given from the beginning). If
$\mathbf{a}=\mathbf{0}$ or $\mathbf{u}=\mathbf{0}$, this is already solved in
Lemma~\ref{lem:caract uP i aQ}(iii) or (ii). Otherwise, write $0\neq \alpha
=\mcd(\mathbf{a})$ and $0\neq \mu =\mcd(\mathbf{u})$; and, according to
Lemma~\ref{lem:caract uP i aQ}, we have to decide whether there exist $\mathbf{a'} \in
\ZZ^m$ and $\mathbf{u'}\in \ZZ^m$, with $\mathbf{a'}\neq \mathbf{0}$, $\alpha \divides
\mcd(\mathbf{a'})$, and $\mu \divides\mcd(\mathbf{u'})$, such that $\mathbf{a'}+\mathbf{u'}
=\mathbf{b}$. Writing $\mathbf{a'}=\alpha \, \mathbf{x}$ and $\mathbf{u'}=\mu \,
\mathbf{y}$, the problem reduces to test whether the following linear system of equations
\begin{equation}\label{eq:sist dem witehead paraules simple}
\left.
\begin{array}{lclcc}
\alpha \, x_1 & + & \mu \, y_1 & = & b_1 \\
 & \vdots &  & \vdots & \\ \alpha \, x_m & + &\mu \, y_m & = & b_m \\
 \end{array}
 \right\}
\end{equation}
has any integral solution $x_1, \ldots,x_m,y_1,\ldots,y_m \in \ZZ$ such that
$(x_1,\ldots,x_m)\neq \mathbf{0}$. A necessary and sufficient condition for the
system~\eqref{eq:sist dem witehead paraules simple} to have a solution is $\mcd(\alpha,
\mu) \divides b_j$, for every $j\in [m]$. And note that, if $(x_1, y_1)$ is a solution to
the first equation, then $(x_1+\mu, y_1-\alpha )$ is another one; since $\mu\neq 0$, the
condition $(x_1,\ldots,x_m)\neq \mathbf{0}$ is then superfluous. Therefore, the answer is
affirmative if and only if $\mcd(\alpha, \mu) \divides b_j$, for every $j\in [m]$; and, in
this case, we can easily reconstruct a monomorphism $\Psi$ of $G$ such that
$(\mathbf{t^a}u)\Psi =\mathbf{t^b}v$.

(i). The argument for automorphisms is completely parallel to the previous discussion
replacing the conditions $\phi$ monomorphism and $\det \mathbf{Q}\neq 0$, to $\phi$
automorphism and $\det \mathbf{Q} =\pm 1$. We manage the first change by using
Theorem~\ref{thm:Whitehead free}~(i) instead of (ii). The second change forces us to look
for solutions to the linear system~\eqref{eq:sist dem witehead paraules simple} with the
extra requirement $\mcd(\mathbf{x})=1$ (because now $\gcd(\mathbf{a'})$ should be equal and
not just multiple of $\alpha$).

So, if any of the conditions $\mcd(\alpha, \mu) \divides b_j$ fails, the answer is negative
and we are done. Otherwise, write $\rho =\mcd(\alpha,\mu)$, $\alpha =\rho \alpha'$ and $\mu
=\rho \mu'$, and the general solution for the $j$-th equation in~\eqref{eq:sist dem
witehead paraules simple} is
 $$
(x_j,y_j) = (x_j^0,y_j^0) + \lambda_j (\mu',-\alpha'), \quad \lambda_j\in \ZZ,
 $$
where $(x_j^0,y_j^0)$ is a particular solution, which can be easily computed. Thus, it only
remains to decide whether there exist $\lambda_1, \ldots, \lambda_m \in \ZZ$ such that
\begin{equation}\label{eq:mcd(sol gen)=1}
\mcd(x_1^0+\lambda_1 \mu' ,\,\ldots ,\, x_m^0 +\lambda_m \mu') =1.
\end{equation}
We claim that this happens if and only if
\begin{equation}\label{eq:mcd(x10,...,xm0,mu)=1}
\mcd(x_1^0,\, \ldots ,\, x_m^0 ,\, \mu')=1,
\end{equation}
which is clearly a decidable condition.

Reorganizing a Bezout identity for~\eqref{eq:mcd(sol gen)=1} we can obtain a Bezout
identity for~\eqref{eq:mcd(x10,...,xm0,mu)=1}. Hence \eqref{eq:mcd(sol gen)=1} implies
\eqref{eq:mcd(x10,...,xm0,mu)=1}. For the converse, assume the integers
$x_1^0,\ldots,x_m^0, \mu'$ are coprime, and we can fulfill equation~\eqref{eq:mcd(sol
gen)=1} by taking $\lambda_1= \cdots = \lambda_{m-1} =0$ and $\lambda_m$ equal to the
product of the primes dividing $x_1^0, \ldots, x_{m-1}^0$ but not $x_m^0$ (take
$\lambda_m=1$ if there is no such prime). Indeed, let us see that any prime $p$ dividing
$x_1^0, \ldots, x_{m-1}^0$ is not a divisor of ${x_m^0+\lambda_m \mu'}$. If $p$ divides
$x_m^0$, then $p$ does not divide neither $\mu'$ nor $\lambda_m$ and therefore $x_m^0
+\lambda_m \mu'$ either. If $p$ does not divide $x_m^0$, then $p$ divides $\lambda_m$ by
construction, hence $p$ does not divide $x_m^0 +\lambda_m \mu'$. This completes the proof
of the claim, and of the theorem for automorphisms.

(iii). In our discussion now, we should take into account endomorphisms of both types.

Again, the argument to decide whether there exists an endomorphism of \ti\ sending
$\mathbf{t^a}u$ to $\mathbf{t^b}v$, is completely parallel to the above proof of (ii),
replacing the condition $\phi$ monomorphism to $\phi$ endomorphism, and deleting the
condition $\det \mathbf{Q}\neq 0$ (and allowing here an arbitrary matrix $\mathbf{Q}$). We
manage the first change by using Theorem~\ref{thm:Whitehead free}~(iii) instead of (ii).
The second change simply leads us to solve the system~\eqref{eq:sist dem witehead paraules
simple} with no extra condition on the variables; so, the answer is affirmative if and only
if $\mcd(\alpha, \mu) \divides b_j$, for every $j\in [m]$.

It remains to consider endomorphisms of \tii, $\Psi_{z,\mathbf{l,h,Q,P}}$. So, given
our elements $\mathbf{t^a}u$ and $\mathbf{t^b}v$, and separating the free and free-abelian
parts, we have to decide whether there exist $z\in F_n$, $\mathbf{l}\in \ZZ^m$,
$\mathbf{h}\in \ZZ^n$, $\mathbf{Q} \in \mathcal{M}_{m}(\ZZ)$, and $\mathbf{P} \in
\mathcal{M}_{n \times m}(\ZZ)$ such that
\begin{equation} \label{eq:Whitehead typeII}
\left. \begin{aligned}
z^{\mathbf{a}\mathbf{l}\tr +\mathbf{u}\mathbf{h}\tr} =v \\
\mathbf{aQ}+\mathbf{uP}=\mathbf{b}
\end{aligned} \ \right\}
\end{equation}
(note that we can ignore the condition $\mathbf{l}\neq \mathbf{0}$ because if
$\mathbf{l}=\mathbf{0}$ then the endomorphism becomes of \ti\ as well, and this case is
already considered before). Again the two equations are independent. About the free part,
note that the integers $\mathbf{a}\mathbf{l}\tr +\mathbf{u}\mathbf{h}\tr$ with $\mathbf{l}
\in \ZZ^m$ and $\mathbf{h} \in \ZZ^n$ are precisely the multiples of
$d=\gcd(\mathbf{a},\mathbf{u})$; so, it has a solution if and only if $v$ is a
$d^{\text{th}}$ power in $F_n$, a very easy condition to check. And about the second
equation, it is exactly the same as when considering endomorphisms of \ti, so its
solvability is already discussed.
\end{proof}

\section*{Acknowledgments}
Both authors thank the hospitality of the Centre de Recerca Matemàtica (CRM-Barcelona) along the research programme on Automorphisms of Free Groups, during which this preprint was finished. We also gratefully acknowledge partial support from the MEC (Spain) through project number MTM2011-25955.
The first named author thanks the support of \emph{Universitat Polit\`{e}cnica de Catalunya} through the PhD grant number 81--727.


\bibliography{mybib}{}

\begin{thebibliography}{10}

\bibitem{artin_algebra_2010}
{\sc Artin, M.}
\newblock {\em Algebra}, 2~ed.
\newblock Addison Wesley, Aug. 2010.

\bibitem{baumslag_intersections_1966}
{\sc Baumslag, B.}
\newblock Intersections of finitely generated subgroups in free products.
\newblock {\em Journal of the London Mathematical Society s1-41\/} (Jan. 1966),
  673--679.

\bibitem{bogopolski_classification_2000}
{\sc Bogopolski, O.}
\newblock Classification of automorphisms of the free group of rank 2 by ranks
  of fixed-point subgroups.
\newblock {\em Journal of Group Theory 3}, 3 (Mar. 2000), 339--351.

\bibitem{bogopolski_introduction_2008}
{\sc Bogopolski, O.}
\newblock {\em Introduction to Group Theory}.
\newblock European Mathematical Society Publishing House, Zurich, Switzerland,
  Feb. 2008.

\bibitem{bogopolski_orbit_2009}
{\sc Bogopolski, O., Martino, A., and Ventura, E.}
\newblock Orbit decidability and the conjugacy problem for some extensions of
  groups.
\newblock {\em Transactions of the American Mathematical Society 362}, 04 (Nov.
  2009), 2003--2036.

\bibitem{bogopolski_basis_2012}
{\sc Bogopolski, O., and Maslakova, O.}
\newblock A basis of the fixed point subgroup of an automorphism of a free
  group.
\newblock {\em {arXiv:1204.6728}\/} (Nov. 2012).

\bibitem{burns_intersection_1998}
{\sc Burns, R., and Kam, S.-M.}
\newblock On the intersection of double cosets in free groups, with an
  application to amalgamated products.
\newblock {\em Journal of Algebra 210}, 1 (1998), 165--193.

\bibitem{ciobanu_monomorphism_2010}
{\sc Ciobanu, L., and Houcine, A.}
\newblock The monomorphism problem in free groups.
\newblock {\em Archiv der Mathematik 94}, 5 (2010), 423--434.

\bibitem{day_peak_2009}
{\sc Day, M.~B.}
\newblock Peak reduction and finite presentations for automorphism groups of
  right-angled artin groups.
\newblock {\em Geometry \& Topology 13\/} (Jan. 2009), 817--855.

\bibitem{day_full-featured_2012}
{\sc Day, M.~B.}
\newblock Full-featured peak reduction in right-angled artin groups.
\newblock {\em {arXiv:1211.0078}\/} (Oct. 2012).

\bibitem{gersten_fixed_1987}
{\sc Gersten, S.}
\newblock Fixed points of automorphisms of free groups.
\newblock {\em Advances in Mathematics 64}, 1 (1987), 51--85.

\bibitem{goldstein_fixed_1986}
{\sc Goldstein, R.~Z., and Turner, E.~C.}
\newblock Fixed subgroups of homomorphisms of free groups.
\newblock {\em Bulletin of the London Mathematical Society 18}, 5 (Sept. 1986),
  468--470.

\bibitem{green_graph_1990}
{\sc Green, E.~R.}
\newblock {\em Graph products of groups}.
\newblock PhD thesis, 1990.

\bibitem{howson_intersection_1954}
{\sc Howson, A.~G.}
\newblock On the intersection of finitely generated free groups.
\newblock {\em Journal of the London Mathematical Society s1-29}, 4 (Oct.
  1954), 428--434.

\bibitem{humphries_stephen_p._representations_1994}
{\sc Humphries, S.~P.}
\newblock On representations of \{A\}rtin groups and the \{T\}its conjecture.
\newblock {\em Journal of Algebra 169}, 3 (1994), 847--862.

\bibitem{kapovich_stallings_2002}
{\sc Kapovich, I., and Myasnikov, A.}
\newblock Stallings foldings and subgroups of free groups.
\newblock {\em Journal of Algebra 248}, 2 (Feb. 2002), 608--668.

\bibitem{laurence_generating_1995}
{\sc Laurence, M.~R.}
\newblock A generating set for the automorphism group of a graph group.
\newblock {\em Journal of the London Mathematical Society 52}, 2 (Oct. 1995),
  318--334.

\bibitem{lyndon_combinatorial_2001}
{\sc Lyndon, R.~C., and Schupp, P.~E.}
\newblock {\em Combinatorial Group Theory}, reprint~ed.
\newblock Springer, Mar. 2001.

\bibitem{makanin_equations_1982}
{\sc Makanin, G.}
\newblock Equations in free groups (russian).
\newblock {\em Izv. Akad. Nauk {SSSR} Ser. Mat. 46\/} (1982), 1190--1273.

\bibitem{marshall_m._cohen_dynamics_1989}
{\sc Marshall M.~Cohen, M.~L.}
\newblock On the dynamics and the fixed subgroup of a free group automorphism.
\newblock {\em Inventiones mathematicae 96}, 3 (1989), 613--638.

\bibitem{maslakova_fixed_2003}
{\sc Maslakova, O.}
\newblock The fixed point group of a free group automorphism.
\newblock {\em Algebra and Logic 42}, 4 (2003), 237--265.

\bibitem{rodaro_fixed_2012}
{\sc Rodaro, E., Silva, P.~V., and Sykiotis, M.}
\newblock Fixed points of endomorphisms of graph groups.
\newblock {\em {arXiv:1210.4094}\/} (Oct. 2012).

\bibitem{sahattchieve_quasiconvex_2011}
{\sc Sahattchieve, J.}
\newblock Quasiconvex subgroups of {$F_m \times \ZZ^n$} are convex.
\newblock {\em {arXiv:1111.0081}\/} (Oct. 2011).

\bibitem{stallings_topology_1983}
{\sc Stallings, J.~R.}
\newblock Topology of finite graphs.
\newblock {\em Inventiones Mathematicae 71\/} (Mar. 1983), 551--565.

\bibitem{turner_finding_1995}
{\sc Turner, E.~C.}
\newblock Finding invisible nielsen paths for a train tracks map.
\newblock {\em Proc. of a workshop held at Heriot-Watt Univ., Edinburg, 1993
  (Lond. Math. Soc. Lect. Note Ser., 204), Cambridge, Cambridge Univ. Press.\/}
  (1995), 300--313.

\bibitem{tyrer_direct_1971}
{\sc Tyrer, J.}
\newblock {\em On Direct products and the Hopf Property}.
\newblock PhD thesis, University of Oxford, 1971.

\bibitem{whitehead_equivalent_1936}
{\sc Whitehead, J. H.~C.}
\newblock On equivalent sets of elements in a free group.
\newblock {\em The Annals of Mathematics 37}, 4 (Oct. 1936), 782--800.
\newblock {ArticleType:} research-article / Full publication date: Oct., 1936 /
  Copyright © 1936 Annals of Mathematics.

\end{thebibliography}
\bibliographystyle{acm}
\addcontentsline{toc}{section}{References}


\section*{} \label{authors}

\begin{minipage}[t]{0.46\textwidth}
\textbf{Jordi Delgado Rodr\'iguez}\hyperref[top]{$^{*}$}

\emph{Dept. Mat. Apl. III,}

\emph{Universitat Polit\`ecnica de Catalunya,}

\emph{Manresa, Barcelona.}

\emph{email}: \texttt{jorge.delgado@upc.edu}
\end{minipage}
\hfill
\begin{minipage}[t]{0.46\textwidth}
\textbf{Enric Ventura Capell}\hyperref[top]{$^{\dag}$}

\emph{Dept. Mat. Apl. III,}

\emph{Universitat Polit\`ecnica de Catalunya,}

\emph{Manresa, Barcelona.}

\emph{email}: \texttt{enric.ventura@upc.edu}
\end{minipage}

\end{document}